\documentclass[11pt]{article}
\usepackage [dvips]{graphics}
\usepackage[centertags]{amsmath}
\usepackage{amsfonts}
\usepackage{amssymb}
\usepackage{amsthm}
\usepackage{newlfont}
\usepackage{stmaryrd}
\usepackage{color,epsfig}
\usepackage{amsfonts,graphicx,psfrag}
\usepackage{graphicx}
\usepackage{float}

\usepackage{txfonts}
\usepackage{mathrsfs}

\theoremstyle{plain}
\newtheorem{thm}{Theorem}[section]
\newtheorem{cor}[thm]{Corollary}
\newtheorem{lem}[thm]{Lemma}
\newtheorem{prop}[thm]{Proposition}

\newtheorem{rem}[thm]{Remark}
\newtheorem{defi}[thm]{Definition}

\def\sqr#1#2{{\vcenter{\vbox{\hrule height.#2pt
              \hbox{\vrule width.#2pt height#1pt \kern#1pt \vrule
width.#2pt}
              \hrule height.#2pt}}}}
\def\be{\begin{equation}}
\def\ee{\end{equation}}

\def\bb{{\beta}}
\def\ga{{\gamma}}

\def\ep{{\epsilon}}

\def\vf{{\varphi}}
\def\vr{{\varrho}}

\def\Sp{{\mathrm {Sp}}}
\def\lb{\label}
\def\bb{{\beta}}
\def\ga{{\gamma}}

\def\R{{\mathbb R}}

\def\U{{\mathbb U}}

\def\no{\noindent}

\def\bs{\bigskip}

\def\dim{\hbox{\rm dim$\,$}}

\def\({\Big (}
\def\){\Big )}
\def\[{\Big[}
\def\]{\Big]}

\def\be{\begin{equation}}
\def\bel{\begin{equation}\label}
\def\ee{\end{equation}}
\def\bea{\begin{eqnarray}}
\def\eea{\end{eqnarray}}
\def\bt{\begin{theorem}}
\def\et{\end{theorem}}
\def\bc{\begin{corollary}}
\def\ec{\end{corollary}}
\def\bl{\begin{lemma}}
\def\el{\end{lemma}}
\def\bp{\begin{proposition}}
\def\ep{\end{proposition}}
\def\br{\begin{remark}}
\def\er{\end{remark}}
\def\ba{\begin{array}}
\def\ea{\end{array}}
\def\bd{\begin{definition}}
\def\ed{\end{definition}}

\makeatletter
   
   \@addtoreset{equation}{section}
\makeatother

\begin{document}

\title{\bf  Collision  Index and Stability of Elliptic Relative Equilibria in Planar $n$-body Problem }
\author{Xijun Hu$^{1}$\thanks{Partially supported
by NSFC(No.11425105, 11131004) and NCET, E-mail:xjhu@sdu.edu.cn }
\quad Yuwei Ou$^{2}$ \thanks{Partially supported
by NSFC(No.11131004) and CPSF(No.2015M580193), E-mail:yuweiou@163.com }
  \\ \\
$^{1}$ Department of Mathematics, Shandong University\\
Jinan, Shandong 250100, The People's Republic of China\\
$^{2}$ Chern Institute of Mathematics, Nankai University \\
Tianjin 300071, The People's Republic of China\\
}
\date{}
\maketitle
\begin{abstract}
It is well known that  a planar central configuration of the $n$-body problem 
gives rise to  solutions where each particle moves on a specific Keplerian orbit 
while the totality of the particles move on a homographic motion. When the  
eccentricity $e$ of the Keplerian orbit  belongs in $[0,1)$,  following Meyer 
and Schmidt,  we call such solutions  {\it elliptic relative equilibria} 
(shortly,  ERE).  In order to study the linear stability of ERE in the 
near-collision case,  namely when $1-e$ is small enough,
we introduce the collision index for planar central configurations.  The 
collision index  is a Maslov-type  index for heteroclinic orbits and 
orbits parametrised by half-lines that, according to the Definition given 
by authors in \cite{HP}, we shall refer to as half-clinic orbits and 
whose Definition in this context, is essentially based on a blow up technique in the 
case $e=1$.  We get the fundamental properties  of collision index and  
approximation theorems.  As applications, we give some new hyperbolic criteria 
and  prove that, generically,  the ERE of minimal central configurations are 
hyperbolic in the near-collision case,  and we  give  detailed analysis of Euler 
collinear orbits  in the  near-collision case.
     \end{abstract}

\bs

\no{\bf AMS Subject Classification:} 37J25, 70F16, 70F10,   37J45, 53D12

\bs

\no{\bf Key Words:} Maslov index,  linear stability, elliptic
relative equilibria,  Euler orbits, planar $n$-body problem

\section{Introduction}

For $n$ particles of mass $m_1,...,m_n$, let
$q_1,...,q_n\in\mathbb{R}^2$ be the position vectors,
$p_1,...,p_n\in\mathbb{R}^2$ be the momentum vectors. Setting 
$d_{i,j}=\parallel q_i-q_j\parallel$, the Hamiltonian function has
the form \bea H=\sum_{j=1}^n\frac{\parallel
p_j\parallel^2}{2m_j}-U(q_1,...,q_n), \,\ U=\sum_{1\leq j< k\leq
n}\frac{m_jm_k}{d_{jk}}. \eea $U$ will be defined on configuration
space $$
\Lambda=\{x=(x_1,\cdots,x_n)\in\mathbb{R}^{2n}\setminus\triangle:
\sum_{i=1}^nm_ix_i=0 \},$$ where
$\triangle=\{x\in\mathbb{R}^{2n}:\exists i\neq j,x_i=x_j \}$ is the
collision set.
 A central configuration is a solution
$(q_1,...,q_n)=(a_1,...,a_n)$ of \bea -\lambda m_jq_j=\frac{\partial
U}{\partial q_j}(q_1,...,q_n) \eea for some constant $\lambda$.  An easy
computation shows that $\lambda=U(a)/\cal{I}(a)>0$ where $\cal{I}(a)=\sum
m_j\|a_j\|^2$ is the moment of inertia.  Otherwise stated,   a
central configuration with $\cal{I}(a)=1$ is a critical point of the function 
$U$
constrained to the set  $\mathcal{E}=\{x\in\Lambda \mid \cal{I}(x)=1 \}$.

It is well known that a planar central configuration of the $n$-body
problem gives rise to solutions  where each particle moves on a
specific Keplerian orbit while the totality of the particles move on
a homographic motion. Following Meyer and Schmidt \cite{MS},  we  call
these solutions as {\it elliptic relative equilibria} and in shorthand notation, 
simply ERE. Specifically,  when the eccentricity $e = 0$, the Keplerian elliptic
motion becomes circular and all the bodies move around
the center of masses along circular orbits with the same frequency.
Traditionally these orbits are called {\it relative equilibria }.

As pointed out in \cite{MS}, there are two four-dimensional
invariant
 symplectic subspaces,  $E_1$ and $E_2$,
and they are associated to the translation symmetry, dilation and
rotation symmetry of the system. In other words, there is a
symplectic coordinate system in which the linearized system of the
planar $n$-body problem decouples into three subsystems on $E_1$,
$E_2$ and $E_3=(E_1\cup E_2)^{\perp}$, where $\perp$ denotes the
symplectic orthogonal complement. A symplectic matrix $\mathcal{M}$
is called spectrally stable if all eigenvalues of $\mathcal{M}$ belongs
to the unit circle $\mathbb{U}$ of the complex plan, while $\mathcal{M}$
is called hyperbolic if no eigenvalues of $\mathcal{M}$ are on
$\mathbb{U}$.  The ERE is called hyperbolic (resp. stable) if the
monodromy matrix $\mathcal{M}$  is restricted to $E_3$, that is,   
$\mathcal{M}|_{E_3}$ is hyperbolic (resp. stable).

There are many interesting results for the linear stability of ERE
(cfr. \cite{M1,M2,R1,R2} and references therein). 
Many of them, investigated  the relative
equilibria for $e$ small enough and as far as we know, only few of them 
studied the linear stability of ERE with $e\in[0,1)$. To our
knowledge, the elliptic Lagrangian solution is the only case that is
well studied. The Lagrangian solution which was discovered by
Lagrange in 1772 \cite{Lag} is the ERE of the equilateral triangle
central configuration in the planar three body problem.

It is well known that the stability of
elliptic Lagrangian solutions depend on the eccentricity $e$ and on
 \be
\beta=\frac{27(m_1m_2+m_1m_3+m_2m_3)}{(m_1+m_2+m_3)^2}.\label{massratio}\ee
Long et al. used Maslov-type index and operator theory to
study the stability problem, and  gave out a full describe of the bifurcation
graph (cfr. \cite{HLS},\cite{HS}). Moreover, Wang et al. built up
a trace formula for linear Hamiltonian systems and Sturm-Liouville
systems, and used it to give an estimate of the stability region as well as 
of the hyperbolic region \cite{HOW},\cite{HO}.

In the study of the near-collision case, that is, when  $1-e$ is small enough, a blow-up 
technique  from R. Mart\'{\i}nez, A. Sam\`{a}, C. Sim\'{o}
\cite{MSS1}  is very useful to carry over our analysis. The authors considered  $4$D linear
system depending on a small parameter $\sigma
>0$ and the singular limit
for $\sigma \rightarrow0$.  Based on it, they
computed the trace $tr_1=\lambda_1+\lambda_1^{-1}$,
$tr_2=\lambda_2+\lambda_2^{-1}$, where $\lambda_i,\lambda_i^{-1}$,
$i=1,2$ are the eigenvalues of monodromy matrix.  Under a
``nondegenerate condition" they describe the asymptotic behaviour of
$\log|tr_i|$, $i=1,2$ and $tr_2$. Their study includes the ERE of Lagrangian
equilateral triangle and Euler collinear central configurations.

Motivated by the these results, we will use blow-up technique and index
theory to study the stability problem of ERE. The index theory we
used will be based on the  Maslov-type index.  The Maslov index
is associated to a  given continuous path of Lagrangian subspaces  and the 
Maslov-type index is assigned to a path of
symplectic matrices. We briefly review the Maslov index theory in \S2.2
and give its relation with the Maslov-type index. For reader's
convenience, we now roughly  describe the Maslov-type  index theory.
Let $\Sp(2n)$ be the set of symplectic matrix in $\mathbb{R}^{2n}$ 
equipped with the 
standard symplectic structure, and set $I_{2n}$
be the identity matrix on $\mathbb{R}^{2n}  $ . Let $\ga\in
C([0,T],\Sp(2n))$ with $\ga(0)=I_{2n}$, for $\omega\in\U$, roughly
speaking,  the Maslov-type index $i_\omega(\ga)$ is the intersection
number (by a small perturbation) of $\ga$ and
$D_\omega:=\{M\in\Sp(2n),\det(M-\omega I_{2n})=0\}$. Please refer to
\cite{Lon4} for the details.

By the blow-up technique, the limit of $ERE$, as $e\to 1$, can be described 
by two heteroclinic orbits $l_0,l_+$  connected
$P_\pm$. (Cfr. to  Figure \ref{fig:side:a1}).  Throughout of the paper, we  
denote by  $\ga_e$  the fundamental solution of the 
essential part of ERE, that is, $\dot{\ga}_e(t)=J\mathcal{B}(t)\ga_e(t), 
t\in[0,2\pi]$, $\ga_e(0)=id$, where $\mathcal{B}(t)$ is defined in   
Equation \eqref{msf}. Given a symmetric matrix $R$ we shall denote by $\lambda_1(R)$, 
the smallest eigenvalue of $R$.
When the  limit  equilibrium $P_\pm$ is  non-hyperbolic, we have the 
following result. 
\begin{thm}\label{thm1.1} Let $a_0\in \mathcal{E}$ be a planar central 
configuration  which satisfies \bea
\lambda_1(D^2U|_\mathcal{E}(a_0))<-\frac{1}{8}U(a_0),\lb{t1}\eea 
where 
$D^2U|_\mathcal{E}(a_0)$ is the Hessian of  $U$ restricted to $\mathcal{E}$ at  
$a_0$. Then 
$i_1(\ga_e)\to\infty$ as $e\to1$. \\ For 
$\frac{1}{U(a_0)}\lambda_1(D^2U|_\mathcal{E}(a_0))=-\frac{1}{8}-r_1$, let 
$\varepsilon=\frac{1}{2}\min \{\frac{r_1}{2r_1+5}, 1/8  \}$,  
$\hat{e}=\frac{1-e^2}{2}$. If  $\hat{e}<\varepsilon^3$, then we have    \bea 
i_1(\ga_e)\geq 
2\frac{\sqrt{r_1}}{\pi}\ln\(\frac{\varepsilon^2}{\sqrt{\hat{e}}}\)-6. \lb{t1.01} 
 \eea
\end{thm}
From \cite{Lon4}, it is well-known that, for any $\omega\in\U$, $|i_\omega-i_1|\leq n$,
Then Theorem 1.1. shows also  that $i_\omega(\ga_e)\to\infty$ as $e\to1$, which
implies there exists a sequence $e_j(\omega)$ converging to $1$, such
that  the system is $\omega$-degenerate.

It is well known that  any $T$-periodic solution  is a critical point of the  
action functional $$ \cal{F}(q)=\int_0^T\[\sum_{i=1}^n 
\frac{m_i\|\dot{q}_i(t)\|^2}{2}+U(q)  \]dt  $$
defined on loop space $W^{1,2}(\mathbb{R}/T\mathbb{Z}, \Lambda)$. Let now 
$x_e$ be the ERE corresponding to $a_0$ with
eccentricity $e$, and let $\phi(x_e)$ be the Morse index of $x_e$ (meaning that 
it is the total number of the negative eigenvalues of
$ \cal{F''}(x_e) $).  Since the Morse index is equal to Maslov-type index (cfr. 
to Lemma \ref{lem5.5}. We have  $\phi(x_e)$  is not less than the 
Maslov-type index $i_1(\ga_e)$ coming from the essential part.  Theorem \ref{thm1.1} implies  
that,  if $a_0$ satisfied
(\ref{t1}), then $\phi(x_e)\to\infty$ as $e\to1$ \cite{HS1}.

The above theorem is related to the result of the interesting paper of
V. Barutello and S. Secchi \cite{BS}. They defined a collision Morse
index for one-collision solution in $n$-body problem with $\alpha$
homogeneous potentials, and proved that the collision index is infinite under the
condition  (\ref{t1}) for the Newton potential. Their results
show that a one-collision solution asymptotic to $a_0$ which satisfied 
(\ref{t1}) cannot be locally minimal for the action function.  A Morse-type 
index theorem both for colliding and parabolic motions, 
will be given in \cite{BHPT}.

The most interesting case, however,  is precisely when $P_\pm$ is hyperbolic, in this case we can
define the Maslov index for heteroclinic orbits and half-clinic
orbits.  (Cfr. Equations \eqref{cl0} and \eqref{clr}).    For {\it half-clinic 
orbits},  we mean a solution  of Hamiltonian system $x(t)$ defined  
on ${\mathbb  R}^+$ or ${\mathbb  R}^-$,  where ${\mathbb  R}^+$ and ${\mathbb R}^-$ stands for the 
non-negative and non-positive half-line, respectively and such that 
the initial condition $x(0)$ belongs to a Lagrangian subspace whilst $x(t)$ converge 
to an equilibrum point  when $t\to\pm\infty$.
Also in this last case , we assign a Maslov  index to both $l_0$ and $l_+$ and 
we shall refer to as {\it collision
index}. After defined the collision index,  we shall prove
Theorem \ref{th.pr} and we shall refer to as approximation theorem.
Let us now show that, under a suitable non degenerate
conditions  for $e\to1$, the Maslov index for $\ga_e$ is convergent to
the sum of collision index on $l_0,l_+$. This is a main part. (Cfr. 
\S3.1 for the details).  In the study of stability
problem, the Dirichlet, Neumann, periodic, anti-periodic boundary
condition play an important role. Our key idea is to use the Maslov
index corresponding to these $4$ kinds of boundary conditions for  determining the
stability. 

Throughout of the paper, we always let $V^j_d=\mathbb{R}^j\oplus 0,
V^j_n=0\oplus\mathbb{R}^j$ be the Lagrangian subspace in $(\mathbb{R}^{2j},\omega_0)$ 
which  corresponding to the Dirichlet and Neumann boundary 
conditions respectively, and we always omit the subscript 
if no confusion is possible. For ERE, by using the approximation theorem, 
we get the following result.
\begin{thm}\lb{thm1.2}  Let  $a_0\in \mathcal{E}$ be such that 
$\lambda_1(D^2U|_\mathcal{E}(a_0))>-\frac{1}{8}U(a_0)$ and we  assume  that $a_0$ is  
nondegenerate and collision nondegenerate. Let $\phi(a_0)$
be the Morse index of $a_0$ which is the total number of negative eigenvalues of  $D^2U|_\mathcal{E}(a_0)$.
 For $1-e$ small enough, we have,  $ \ga_e(2\pi)V_d\pitchfork V_d$,
\bea \mu(V_d,\ga_e(t) V_d, t\in[0,2\pi]
)=k+i(V_d;l_+),\lb{prd}\eea and $ \ga_e(2\pi)V_n\pitchfork
V_n$, \bea \mu(V_n,\ga_e(t) V_n, t\in[0,2\pi]
)=2\phi(a_0)+i(V_d;l_+),\lb{prn}\eea where $\pitchfork$ means
transversal, $k=2n-4$,  $i(V_d;l_+)$ is the collision index on $l_+$ defined by (\ref{clr}) and $\mu$ is the
Maslov index.
\end{thm}
The definition of collision  nondegenerate index is given in
Definition \ref{defcn}. 
The degenerate problem  on $l_0$ will be discussed in \S3.2. We observe that, in contrast 
with respect to the nondegeneracy condition along $l_0$, we didn't establish a useful 
criterion for detecting the nondegeneracy along $l_+$.

If the central configurations have brake symmetry (cfr. Definition
\ref{defs}),  the  collision index of heteroclinic
could be decomposed into the sum of index on half-clinic orbits and this 
will simplify the computation.  To our knowledge, the
Lagrangian and Euler central configurations both have brake symmetry.
Another example is the $1+n$ central configurations, that is regular
polygon configurations with a central mass. It will be interesting to 
provide central configurations without this symmetry property.

As an application, we study the stability of ERE for minimizer 
central configurations. For a central configuration $a_0$,  it is obvious that $D^2U|_\mathcal{E}(a_0)$
has a trivial eigenvalue $0$ which comes from the rotation invariant.
  The central configuration $a_0$ is called
nondegenerate minimizer if all the nontrivial eigenvalues are bigger than $0$,
while $a_0$ is
called strong minimizer if all the nontrivial eigenvalues are bigger than $U(a_0)$.
\begin{thm}\label{thm1.3}
We assume that  $a_0$ is a nondegenerate minimizer that satisfies the collision
nondegenerate condition.  If  $1-e$ is sufficiently small, then the  ERE  is
 hyperbolic.
\end{thm}
In the case $e=0$, Moeckel conjectured \cite{Ab} that a relative  equilibrium is 
linearly stable  only if it
associated to a minimizing  central configuration. 
Our results show that   in the case that $1-e$  small enough, 
it is generally  hyperbolic. By the way, we
conjecture that Theorem \ref{thm1.3} is true also without  the
collision nondegenerate condition.

In the case $a_0$ strong minimizer the following result holds.
\begin{thm}\label{thm1.4}
 The ERE of  a  strong minimizer $a_0$  is hyperbolic for any
 $e\in[0,1)$.
\end{thm}
A typical example of nondegenerate minimizer central configurations
is the Lagrangian central configurations, which is strong minimizer if
$\beta>8$. For Lagrangian orbits, the  conclusion of Theorem
\ref{thm1.3} was proved in \cite{HLS} without the collision
nondegenerate condition and the result of Theorem \ref{thm1.4}  was
proved by Ou \cite{Ou}. Another easy example is the $1+3$-gon central
configurations, that is,  the regular triangular configurations with a
central mass. The three unit masses with unit distance away from the
mass $m_c$ at the origin. As a direct application of Theorem \ref{thm1.3} and \ref{thm1.4},
we get the next result.
\begin{cor}\lb{cor1.5} Let 
$a_0$ be a $1+3$-gon central configuration with central mass $m_c$,  
for  $m_c\in[0,\frac{81+64\sqrt{3}}{249})$ and we assume that 
$a_0$ is collision nondegenerate. If  $1-e$ is sufficiently small, 
then the  ERE is hyperbolic
Furthermore, if $m_c\in[0,\frac{\sqrt{3}}{24})$ then  the ERE is hyperbolic for any
$e\in[0,1)$.
\end{cor}
Another conjecture of Moeckel \cite{Ab} states that a relative equilibrium 
is linearly stable only if it has a dominant mass.  For example  
the Lagrangian orbits and ERE of  $1+3$-gon relative to a strong minimizer 
have no dominant mass. Thus Theorem \ref{thm1.4} can be considered  
as a support of Moeckel's conjecture  in the case of $e>0$, so we guess  
Moeckel's conjecture is also true in the case of ERE.

The collision index plays an important role in the study of the stability
problem. We shall give some conjectures for the collision index which are
related with Y. Long's conjecture for the Maslov-type index of ERE. (Cfr.
Remark \ref{r5.1} for further details).

As a further  application, we consider the ERE of Euler collinear central
configurations \cite{Eu}, that  we simply refer to as elliptic Euler orbits.
The linear stability of this kind of orbits  depends upon two parameters, 
$e$ and $\delta$, where the last one $\delta\in[0,7]$ only 
depends on mass $m_1,m_2,m_3$.  (Cfr. Appendix A of \cite{MSS1} 
and \cite{LZhou} for further details). To
our knowledge, the near collision case was firstly studied by  R.
Mart\'{\i}nez, A. Sam\`{a}, C. Sim\'{o} \cite{MSS1}. Y. Long and
Q. Zhou used  Maslov-type index theory in order to describe the
$\pm1$-degenerate curves; they also  analysed the stability
problem. It is worth noting that by their methods is not possible to 
explain  the limit property of
$\pm1$-degenerate curves  numerically proved by   R. Mart\'{\i}nez,
A. Sam\`{a}, C. Sim\'{o} \cite{MSS1}.  Please refer to Figure \ref{fig:side:a4} 
and Figure \ref{fig:side:b4}.  Using
the collision index, we explain the limit property. We show that
$\delta>1/8$ is equivalent  to condition (\ref{t1}).  Theorem
\ref{thm1.1} implies the $\pm1$-degenerate curves  don't  intersect
$[1/8,7]\times 1$. In the case, $\delta\in(0,1/8)$, the collision
index is well defined, we analysis the near-collision phenomena by
the collision index. We  can compute in detail for collision index
on $l_0$,  but unfortunately, we can't determine  the collision index on $l_+$ by
analytical method.  Instead, we develop a numerical method to compute
the collision index. Based on numerical results A, the collision
index strictly proved the behaviour of the $\pm1$ in the near-collision case. Please refer to 
\S5.2 for the details.

This paper is organized as follows. We review the Meyer-Schmidt
reduction and Mart\'{\i}nez, Sam\`{a},Sim\'{o}  blow up technique at
\S2.1. We give a brief introduction of the   Maslov index theory and
we prove Theorem \ref{thm1.1} in \S2.2. The definition of collision
index is stated and the approximation theorem  is proved in \S3.1.   Some
basic properties of the collision index are given in \S3.2. In \S 4, we
study the case of brake symmetric central configurations.    The
computation of the collision index along $l_0,l_+$ is given in \S3.2, \S3.3.
We give some applications in \S5.  In \S5.1, we study the minimizing  central
configurations and we prove Theorem \ref{thm1.3} and Theorem
\ref{thm1.4}.  We use the collision index to analyse  the
Euler orbits in \S5.2. At last, for the reader's convenience, 
we give the details of the numerical method used to compute collision 
index in \S6.

\section{Blow up and limit index for the non-hyperbolic case }\label{sec2}

This section includes some basic preliminaries.  We first briefly review
the decomposition of ERE by following  authors in 
\cite{MS} and the blow-up technique of Mart\'{\i}nez, Sam\`{a} and
Sim\'{o} \cite{MSS2} in  section \S 2.1,  then  we review the
fundamental property of Maslov index in \S 2.2, and   give
the proof of Theorem \ref{thm1.1}.

\subsection{Reduction and   blow up method}

In 2005, Meyer and Schmidt strongly used the structure of the central configuration for 
the elliptic Lagrangian orbits and symplectically decomposed the fundamental 
solution of the elliptic Lagrangian orbit into two parts, 
one of which corresponding to the Keplerian solution and the other is the essential part
of the dynamics, needed for studying the  stability.
   For the reader's convenience, we briefly review the 
   central configuration coordinates, by following Meyer and Schmidt \cite{MS}.

 Suppose that  $\mathcal{Q}=(q_1,...,q_n)\in \mathbb{R}^{2n}$  
   with mass $m_1,...,m_n$ is a central configuration, and $\mathcal{P}=(p_1,...,p_n)\in\mathbb{R}^{2n}$.
Let $I_j$ be the identity matrix on $\mathbb{R}^j$, $J_{2j}=\left( \begin{array}{cccc}0_j& -I_j \\
I_j & 0_j \end{array}\right)$. 
We denote by $\mathbb{J}_n=diag(J_2,...,J_2)_{2n\times2n} $ and  
$M=diag (m_1,m_1,m_2,m_2,...,m_n, m_n)_{2n\times2n}$.
We assume that $t \mapsto x(t)$ is a periodic solution of ERE, then the corresponding fundamental solution is
\bea \dot{\ga}(t)=J_{4n}H''(x(t))\ga(t),\,\ \ga(0)=I_{4n}. \label{ga}  \eea
As in \cite[Corollary 2.1, pag.266]{MS}, Equation (\ref{ga}) can be decomposed into $3$ subsystems on
$E_1$, $E_2$ and $E_3=(E_1\cup E_2)^{\perp}$ respectively. The basis
of $E_1$ is $(0,u)$, $(Mu,0)$, $(0,v)$, $(Mv,0)$, where $u =(1, 0,
1, 0, ...)$, $v =(0,1, 0, 1, .)$,
 and $E_2$ is spanned by
$(0,\mathcal{Q})$, $(M\mathcal{Q},0)$, $(0,\mathbb{J}_n\mathcal{Q})$, $(\mathbb{J}_nM\mathcal{Q},0)$.
For $X = (g, z,w)\in \mathbb{R}^2\times\mathbb{R}^2\times\mathbb{R}^{2n-4}$ and $Y =
(G,Z,W)\in\mathbb{R}^2\times\mathbb{R}^2\times\mathbb{R}^{2n-4}$, we consider the linear
symplectic transformation of the form $\mathcal{Q} = PX, \mathcal{P} = P^{-T} Y$, where
$P$ is such that $\mathbb{J}P = P\mathbb{J}, P^TMP = I_{2n}$ (\cite{MS}, p263).
Now $B(t)=H''(x(t))$ in this new coordinate system has the form
$B(\mathcal{Q})=B_1\oplus B_2\oplus B_3$, where $B_i=B|_{E_i}$.
The essential
part $B_3(t)$  is a path of  $(4n-8)\times(4n-8)$ symmetric matrix.

In the rotating coordinate system and by using the true anomaly as the variable, 
Meyer and Schmidt \cite{MS} gave a very useful form of the essential part
\bea  \mathcal{B}(t)=\left( \begin{array}{cccc} I_{k} & -\mathbb{J}_{k/2} \\
\mathbb{J}_{k/2} & I_{k}-\frac{I_{k}+\mathcal {D}}{1+e\cos(t)}
\end{array}\right),\quad t\in[0,2\pi], \label{msf} \eea where  $k=2n-4$ and $e$ is
the eccentricity, $t$ is the true anomaly and \bea \mathcal
{D}=\frac{1}{\lambda}P^TD^2U(\mathcal{Q})P\big|_{w\in\mathbb{R}^{k}},\quad
with \quad \lambda=\frac{U(\mathcal{Q})}{\cal{I}(\mathcal{Q})}.  \eea We
denote by $R:=I_{k}+\mathcal {D}$, which  can be considered as the
regularized Hessian of the central configurations.
 In fact, for $a_0\in\mathcal{E}$ which is a central configurations,
 then $\cal{I}(a_0)=1$. With respect to the mass matrix $M$
inner product, the Hessian of the restriction of the potential to the inertia 
ellipsoid, is given by 
$$D^2U|_\mathcal{E}(a_0)=M^{-1}D^2U(a_0)+U(a_0).$$
Then we have \bea P^{-1}D^2U|_\mathcal{E}(a_0)
P=P^TD^2U(a_0)P+U(a_0), \eea and thus \bea
R=\frac{1}{U(a_0)}P^{-1}D^2U|_\mathcal{E}(a_0)P\big|_{w\in\mathbb{R}^{k}}  \,\ .\lb{R} \eea 
Thus the corresponding Sturm-Liouville system  is
  \bea -\ddot{y}-2\mathbb{J}_{k/2}\dot{y}+\frac{R}{1+e\cos(t)}y=0.  \label{st} \eea

In order to study the singular limit case $e\to1$, we use a change of coordinates as in 
\cite{MSS2}. In fact, in the case of Newtonian potential, this can be interpreted as a 
McGehee change of coordinates (cfr. \cite{Mc}, \cite{M0}).
Let $q=(1+e\cos(t))^{1/2}$, $Q=-2\dot{q}$ and change the time variable to $\tau$, where
$dt=qd\tau$.  Throughout the paper,  we always use $x'=
\frac{dx}{d\tau}$ and $\dot{x}=\frac{dx}{dt}$.  Then we have
\bea q'=-\frac{1}{2}qQ,\quad  Q'=\frac{1}{2}Q^2+q^2-1.  \lb{bp} \eea
We observe that
(\ref{bp}) is well defined for $q=0$ and its first integral   is
$E=q^2(\frac{Q^2}{2}+\frac{q^2}{2}-1)$.  An easy computation shows that, 
for the orbits with eccentricity $e$, the first integral with $E=\frac{e^2-1}{2}=-\hat{e}$.
The system has two
equilibria $P_\pm=(0,\pm\sqrt{2})$ lying on the level set $E=0$.  We
distinguish the level set $E=0$ into two orbits \bea l_0=\{(q,Q)\in
\mathbb{R}^2| q=0,|Q|<\sqrt{2}  \}, \eea and \bea l_+=\{(q,Q)\in
\mathbb{R}^2|q>0,  Q^2+q^2=2  \}. \eea On
$l_0$,  we have \bea  q_{l_0}(\tau)=0, \quad
Q_{l_0}(\tau)=-\sqrt{2}\tanh(\frac{\sqrt{2}}{2}\tau),  \eea and the
system on $l_+$ is \bea  q'_{l_+}=-\frac{1}{2}qQ, \quad
Q'_{l_+}=-\frac{q^2}{2}.  \eea The solution is given by \bea
q_{l_+}(\tau)=\sqrt{2}/\cosh(\frac{\sqrt{2}\tau}{2}), \quad
Q_{l_+}(\tau)=\sqrt{2}\tanh(\frac{\sqrt{2}\tau}{2}). \eea For
convenience, we also let $l^-_0=\{(p,Q)\in l_0, Q\leq0\}$ and
$l^+_0=\{(p,Q)\in l_0, Q\geq0\}$, similarly, let $l^-_+=\{(p,Q)\in
l_+, Q\leq0\}$ and $l^+_+=\{(p,Q)\in l_+, Q\geq0\}$. Obviously,
$l_0=l_0^-\cup l_0^+$ and $l_+=l_+^-\cup l_+^+$.

Throughout of the paper, we always let $\ga_e$ ($e\in[0,1)$) be the fundamental solution of (\ref{msf}). 
For simplicity, $\ga_e$ can also be considered as a
function of $\tau$. 
For $q\neq0$ we consider the matrix 
$S=diag(q^{\frac{1}{2}}I_k,q^{-\frac{1}{2}}I_k)\in\Sp(2k)$ and for $\mathcal{T}=\tau(2\pi)$,
we have $S(\mathcal{T})=S(0)$. 
We let $\hat{\ga}_e(\tau)=S(\tau)\ga_e(\tau)S^{-1}(0)$ and we observe that the associated monodromy matrix is similar to the one of  $\ga_e$. 
A direct computation shows that
\bea \frac{d}{d\tau}\hat{\ga}_e=J\hat{B}\hat{\ga}_e, \quad \hat{\ga}_e(0)=I_{2k},  \quad\tau\in[0,\mathcal{T}],  
\lb{rg} \eea
with \bea \hat{B}= \left( \begin{array}{cccc} I_{k} & \frac{Q}{4}I_k-q\mathbb{J}_{k/2} \\
\frac{Q}{4}I_k+q\mathbb{J}_{k/2} & q^2I_{k}-R \end{array}\right).
\eea

\begin{figure}[H]
\begin{minipage}[t]{0.5\linewidth}
\centering
\includegraphics[height=1.2\textwidth,width=0.696\textwidth]{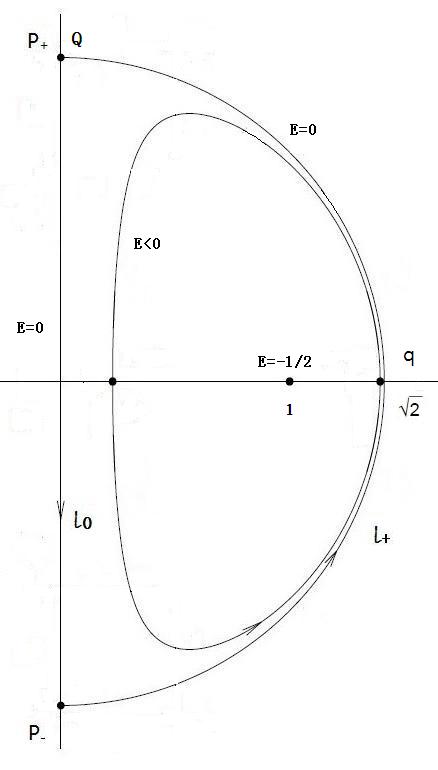}
     \caption{Phase portrait of (\ref{bp}) from \cite{MSS2}.}
\label{fig:side:a1}
\end{minipage}%
\begin{minipage}[t]{0.5\linewidth}
\centering
\includegraphics[height=1.2\textwidth,width=0.696\textwidth]{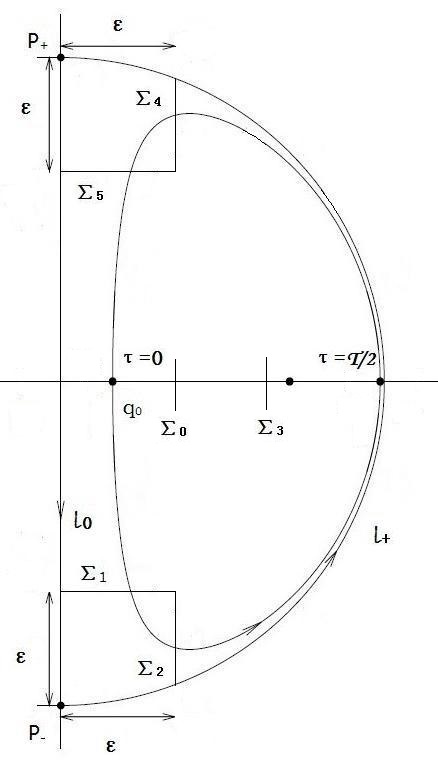}
     \caption{An illustration of the sections used in the proof of Theorem \ref{th.pr}. }
\label{fig:side:b1}
\end{minipage}
\end{figure}

The linear system (\ref{rg}) is well-defined also when $e=1$. In
this case, $E=0$, the system has two equilibria corresponding to
$P_\pm$, and the system can be considered as two heteroclinic orbits.
\begin{prop}\label{prop2.1}  $P_\pm$ is hyperbolic if and only if $\lambda_1(R)>-\frac{1}{8}$.  \end{prop}
\begin{proof} 
We observe that   at points $P_\pm$, the linear part
with form
 $D_\pm=J_k\left( \begin{array}{cccc} I_{k} & \pm\frac{\sqrt{2}}{4}I_k \\
\pm\frac{\sqrt{2}}{4}I_k & -R
\end{array}\right) $.  $P_\pm$ is hyperbolic if and only if the eigenvalue of
$D_\pm$ is not on the imaginary line. Since $R$ is diagonalizable by choose suitable bases, the
results is from simple computations.
\end{proof}

Given $\varepsilon<1/8$, we define the
following sections (See figure 2)
\bea \Sigma_{0}=\{(q, Q) |0<q<\varepsilon, Q=0\}, \ \ \ \Sigma_{1}=\{(q, Q)|0<q<\varepsilon, 
Q=-\sqrt{2}+\varepsilon\}, \nonumber \\
\Sigma_{2}=\{(q, Q) |q=\varepsilon, -\sqrt{2-\varepsilon^2}<Q<\varepsilon-\sqrt{2}\},
\ \ \ \Sigma_{3}=\{(q, Q)|0<\sqrt{2}-q<\varepsilon, Q=0\}.
\nonumber
\\ \Sigma_{4}=\{(q, Q) |q=\varepsilon, \sqrt{2}-\varepsilon<Q<\sqrt{2-\varepsilon^2}\},\ \
\ \ \Sigma_{5}=\{(q, Q)|0<q<\varepsilon,
Q=\sqrt{2}-\varepsilon\}. \nonumber \eea

If  $\hat{e}=\frac{1-e^2}{2}$, $\tilde{E}(\hat{e}):=\{(q,Q), E=-\hat{e}\}$ denotes 
the energy level set. A direct computation
show that $\tilde{E}(\hat{e})$ intersects $\Sigma_i$, $i=0,\cdots,5$ simply, i.e. 
intersect  at exactly one point  when 
$\hat{e}<\varepsilon^3$, $\varepsilon<1/8$.  
In this case,  the Poincar\'{e} map $\mathcal{P}_{i}: \Sigma_{i-1}\mapsto
\Sigma_{i}$, $i=1,\cdots,5$ is well defined.  
In fact, $\hat{e}<\varepsilon^2$ ensures the intersection with $\Sigma_0, \Sigma_3$.
For  the intersection of $\tilde{E}(\hat{e})$ and $\Sigma_1$,  
we observe that, since  $Q=-\sqrt{2}+\varepsilon$, are solutions of  
$ q^2((-\sqrt{2}+\varepsilon)^2/2+q^2/2-1)=-\hat{e}$,
we get $q^2_\pm=\frac{1}{2}(\varepsilon(2\sqrt{2}-\varepsilon)\pm((\varepsilon(2\sqrt{2}
-\varepsilon))^2-8\hat{e})^{\frac{1}{2}})$. 
For $\varepsilon<1/8$,  $\hat{e}<\varepsilon^3$, we have
 $q^2_-<\varepsilon^2<q^2_+$, which guarantee the simple intersection.  
 Since $\tilde{E}(\hat{e})$ is convex, this also
ensure the  simple intersection with $\Sigma_2$, and we have  the simple 
intersection of $\Sigma_4, \Sigma_5$ by symmetry.

   Following \cite{MSS2}, we denote by $\tau_{l_{0}}>0$ the time defined 
   by $Q_{l_{0}}(\tau_{l_{0}})=-\sqrt{2}+\varepsilon$ and 
   $\tau_{l_{+}}>0$ such that $q_{l_{+}}(-\tau_{l_{+}})=\varepsilon$.
It is obvious $Q_{l_{0}}(-\tau_{l_{0}})=\sqrt{2}-\varepsilon$ and
$q_{l_{+}}(\tau_{l_{+}})=\varepsilon$.  $\tau_{l_{0}}$ and
$\tau_{l_{+}}$ are finite and independent of $e$
once $\varepsilon$ is fixed.
 Let
$q_{0}=q(0)$ and $\tau_{1}$, $\tau_{2}$ be the smallest positive
time such that $(q(\tau_{1}), Q(\tau_{1}))\in\Sigma_{1}$ and
$(q(\tau_{2}), Q(\tau_{2}))\in\Sigma_{2}$ . It is clear that
$q_{0}$, $\tau_{1}$ and $\tau_{2}$ depend on $\hat{e}$. Moreover
$q_{0}\rightarrow0$, $\tau_{1}\rightarrow\tau_{l_{0}}$ and
$\mathcal{T}/2-\tau_{2}\rightarrow\tau_{l_{+}}$ when
$\hat{e}\rightarrow0$. Similarly,
 let $\tau_{4}$, $\tau_{5}$ be the
smallest positive time such that $(q(\tau_{4}),
Q(\tau_{4}))\in\Sigma_{4}$ and $(q(\tau_{5}),
Q(\tau_{5}))\in\Sigma_{5}$. We have $\tau_4-\mathcal{T}/2\to
\tau_{l_{+}}$ and $\mathcal{T}-\tau_5\to \tau_{l_{0}} $.
The next lemma was proved in\cite{MSS2}; however, for the reader's convenience, 
we proved it with a slight sharp estimates.
\begin{lem}  For $\varepsilon<1/8$, $\hat{e}<\varepsilon^3$, we have \\
(a) \bea \sqrt{2}\ln(\frac{\varepsilon}{q(\tau_1)})\leq \tau_2-\tau_1
\leq \frac{2}{\sqrt{2}-\varepsilon}\ln(\frac{\varepsilon}{q(\tau_1)}), \lb{adde1} \eea
(b) \bea \int_{\tau_1}^{\tau_2}q(\tau)d\tau\leq\frac{2\varepsilon}{\sqrt{2}-\varepsilon}<2\varepsilon,
\lb{adde2} \eea
(c) \bea \int_{\tau_1}^{\tau_2}|Q(\tau)+\sqrt{2}|d\tau<2\varepsilon. \lb{adde3} \eea
\end{lem}
\begin{proof}  To prove (a), we observe that for $\tau\in[\tau_1,\tau_2]$,  
$-\sqrt{2}\leq Q(\tau)\leq -\sqrt{2}+\varepsilon$.
Multiplying by $-q(\tau)/2$ we get
$\frac{1}{2}(\sqrt{2}-\varepsilon)q(\tau)\leq q'(\tau)\leq \frac{\sqrt{2}}{2}q(\tau)  $ and 
hence Equation (\ref{adde1}) by a direct integration.  To prove (b),  please note that
$\int_{\tau_1}^{\tau_2}q(\tau)d\tau\leq \frac{2}{\sqrt{2}-\varepsilon}\int_{\tau_1}^{\tau_2}
q'(\tau)d\tau\leq \frac{2}{\sqrt{2}-\varepsilon}q(\tau_2)$. Equation  (\ref{adde2}) 
follows by observing that $q(\tau_2)=\varepsilon$.

To prove (c), let $F(q)=\sqrt{2}-\sqrt{2-q^2}$. Then  
$F(q)\leq cq$ for $q<\varepsilon$,  with $c=\frac{\varepsilon}{\sqrt{2}+\sqrt{2-\varepsilon^2}}$. 
From  (b), we have \bea \int_{\tau_1}^{\tau_2}F(q)d\tau \leq \frac{2c\varepsilon}{\sqrt{2}-\varepsilon}.
\lb{adde4}\eea   
Please observe that, near $P_-$,  $l_+$ is graph of $F(q)-\sqrt{2}$. Let $y=Q+\sqrt{2}-F(q)$, 
then $0<y<\varepsilon$.  By a direct 
computation it follows  that \bea  y'=(\frac{q^2}{2\sqrt{2-q^2}}-\sqrt{2-q^2})y+y^2/2 = -\sqrt{2}y(1+o_1), \nonumber \eea
where $|o_1|\leq\varepsilon $. Thus we have $-\sqrt{2}y(1+\varepsilon)\leq y'\leq-\sqrt{2}y(1-\varepsilon)$, then 
\bea \int_{\tau_1}^{\tau_2}yd\tau \leq \frac{1}{\sqrt{2}(1-\varepsilon)}y(\tau_1)\leq \frac{\varepsilon}{\sqrt{2}
(1-\varepsilon)}.
\lb{adde5}\eea
From (\ref{adde4}, \ref{adde5}), we have \bea \int_{\tau_1}^{\tau_2}|Q+\sqrt{2}|d\tau \leq \int_{\tau_1}^{\tau_2}(y+F(q))d\tau\leq 
\frac{\varepsilon}{\sqrt{2}(1-\varepsilon)}+\frac{2c\varepsilon}{\sqrt{2}-\varepsilon}<2\varepsilon. \lb{adde6}\eea
\end{proof}

\subsection{The index limit  on the non-hyperbolic case  }\label{subsec3.1}

We  start by briefly reviewing the  Maslov index theory
\cite{Ar,CLM, RS1} in this subsection.  Let $(\mathbb{R}^{2n},\omega)$ be the standard
symplectic space, and $Lag(2n)$ the Lagrangian Grassmanian, i.e. the set of
Lagrangian subspaces of $(\mathbb{R}^{2n},\omega)$.   For two 
continuous paths $L_1(t),L_2(t)$, $t\in[a,b]$ in $Lag(2n)$,
the Maslov index $\mu(L_1,L_2)$ is an integer invariant.   
Here we use the definition from \cite{CLM}.
We list several properties of the Maslov index. The details could be found
in \cite{CLM}.

Property I (Reparametrization invariance)  Let
$\vr:[c,d]\rightarrow [a,b]$ be a continuous and piecewise smooth
function with $\vr(c)=a$, $\vr(d)=b$, then \bea \mu(L_1(t),
L_2(t))=\mu(L_1(\vr(\tau)), L_2(\vr(\tau))). \lb{adp1.1} \eea

Property II (Homotopy invariant with end points)  For two continuous
family of Lagrangian path $L_1(s,t)$, $L_2(s,t)$, $0\leq s\leq 1$, $a\leq
t\leq b$, and satisfies  $dim L_1(s,a)\cap L_2(s,a)$ and $dim
L_1(s,b)\cap L_2(s,b)$ is constant, then  \bea \mu(L_1(0,t),
L_2(0,t))=\mu(L_1(1,t),L_2(1,t)). \lb{adp1.2} \eea

Property III (Path additivity)  If $a<c<b$, then then 
\bea \mu(L_1(t),L_2(t))=\mu(L_1(t),L_2(t)|_{[a,c]})+\mu(L_1(t),L_2(t)|_{[c,b]}).
\lb{adp1.3} \eea

Property IV (Symplectic invariance)  Let $\ga(t)$, $t\in[a,b]$ is a
continuous path in $\Sp(2n)$, then \bea \mu(L_1(t),L_2(t))=\mu(\ga(t)L_1(t), \ga(t)L_2(t)). \lb{adp1.4} \eea

Property V (Symplectic additivity)  Let $W_i$, $i=1,2$ be symplectic space, 
$L_1,L_2\in C([a,b], Lag(W_1))$ and $\hat{L}_1,\hat{L}_2\in C([a,b], Lag(W_2))$, 
then \bea \mu(L_1(t)\oplus \hat{L}_1(t),L_2(t)\oplus \hat{L}_2(t))= \mu(L_1(t),L_2(t))+
\mu(\hat{L}_1(t),\hat{L}_2(t)). 
\lb{adp1.5add} \eea

In the case $L_1(t)\equiv V_0$,  $L(t)=\ga(t)V$, where $\ga$ is a path of  symplectic matrix we have 
a monotonicity property (cfr. \cite{HOW}).

Property VI (Monotone property)   Suppose for $j=1,2$,   $L_j(t)=\ga_j(t)V$, 
where  $\dot{\ga}_j(t)=JB_j(t)\ga_j(t)$ with $\ga_j(t)=I_{2n}$. 
If $B_1(t)\geq B_2(t)$ in the sense that $B_1(t)-B_2(t)$ is non-negative matrix, 
then for any $V_0,V_1\in Lag(2n)$, we have
\bea  \mu(V_0, \ga_1V_1)\geq  \mu(V_0, \ga_2V_1).    \lb{adp1.5monotone}       \eea

One efficient way to study the Maslov
index is via crossing form introduced by \cite{RS1}.  
 For simplicity and since it is  enough for our purpose, we only
review the case of the Maslov index for a path of Lagrangian subspace
with respect to a fixed Lagrangian subspace. Let $\Lambda(t)$ be a
$C^1$-curve of Lagrangian subspaces with $\Lambda(0)=\Lambda$, and
let $V$ be a fixed Lagrangian subspace which is transversal to
$\Lambda$. For $v\in \Lambda$ and small $t$, define $w(t)\in V$ by
$v+w(t)\in \Lambda(t)$. Then the form \bea
Q(v)=\left.\frac{d}{dt}\right|_{t=0}\omega(v,w(t)) \lb{1.3a} \eea is
independent of the choice of $V$ (cfr.\cite{RS1}).  A crossing for
$\Lambda(t)$ is some $t$ for which $\Lambda(t)$ intersects $W$
nontrivially, i.e. for which $\Lambda(t)\in\overline{O_1(W)}$. The
set of crossings is compact. At each crossing, the crossing form is
defined to be \bea \Gamma(\Lambda(t),W,t)=Q|_{\Lambda(t)\cap W}.
\lb{1.3b} \eea A crossing is called {\it regular} if the crossing
form is non-degenerate. If the path is given by
$\Lambda(t)=\gamma(t)\Lambda$ with $\gamma(t)\in \Sp(2n)$ and
$\Lambda\in Lag(2n)$, then the crossing form is
equal to $(-\gamma(t)^TJ\dot{\gamma}(t)v,v)$, for $v\in
\gamma(t)^{-1}(\Lambda(t)\cap W)$, where $(\,,\,)$ is the standard
inner product on $\mathbb{R}^{2n}$.

For $\Lambda(t)$  and $W$ as before, if the path has only regular
crossings, following \cite{LZ}, the Maslov index is equal to \bea
\mu(W,\Lambda(t))=m^+(\Gamma(\Lambda(a),W,a))+\sum_{a<t<b}  sign
(\Gamma(\Lambda(t),W,t))-m^-(\Gamma(\Lambda(b),W,b)),\lb{1.3c}\eea
where the sum runs  all over the  crossings $t\in(a,b)$ and $m^+,
m^-$ are the dimensions of  positive and negative definite
subspaces, $sign=m^+-m^-$ is the signature. We note that for a
$C^1$-path $\Lambda(t)$ with fixed end points, and we can make it only
have regular crossings by a small perturbation.

In contrast with the definition given in equation 
(\ref{1.3c}), the Maslov index defined in \cite{RS1} has the following form
\bea
\mu_{RS}(\Lambda(t),W)=\frac{1}{2}sign(\Gamma(\Lambda(a),W,a))+\sum_{a<t<b}  sign
(\Gamma(\Lambda(t),W,t))+\frac{1}{2}sign(\Gamma(\Lambda(b),W,b)).\lb{rsm}\eea
We observe that,  for the non-degenerate path (i.e. $L(t)\cap W=0$ for $t=a,b$),
\bea \mu(W,L(t))=\mu_{RS}(L(t),W).  \nonumber \eea

Note that for $M\in \Sp(2n)$, $Gr(M):=\{(x,Mx)\,\,|\,\,
x\in\mathbb{R}^{2n}\}$ is a Lagrangian subspace of the symplectic vector
space $(\mathbb{R}^{2n}\oplus\mathbb{R}^{2n},-\omega\oplus\omega)$. 
Let $\gamma(t)$ be a path of symplectic matrices,
$\Lambda=\Lambda_1\oplus \Lambda_2\in  Lag(4n)$, where
$\Lambda_i\in Lag(2n)$, for $i=1,2$, then following
\cite{RS1} and  by computing the crossing forms, we have \bea
\mu(\Lambda_1\oplus
\Lambda_2,Gr(\gamma(t)))=\mu(\Lambda_2,\gamma(t)\Lambda_1). \lb{1r2}
\eea

For a continuous path $\ga(t)\in\Sp(2n)$ with $\ga(0)=I_{2n}$, the
Maslov-type index $i_\omega(\ga)\in \mathbb  Z$  is a very useful tool
in studying the periodic orbits of Hamiltonian systems \cite{Lon4}. The
next lemma  ( \cite{LZ} Corollary 2.1.)  gives its relation with
the Maslov index.
\begin{lem}\lb{lem2.2}For any $\gamma(t)$, we have
  \be
i_1(\gamma)+n=\mu(\triangle, Gr(\gamma (t))), \lb{1.7} \ee and \be
i_\omega(\gamma)=\mu(Gr(\omega), Gr(\gamma
(t))),\omega\in\mathbb  U\backslash\{ 1\}, \lb{1.7.1} \ee where $\Delta$ is
the diagonal $Gr(I_{2n})$, $Gr(\omega)=Gr(\omega I_{2n})$.
\end{lem}

For $V_1,V_2\in Lag(2n)$ and a Lagrangian path $t \mapsto\Lambda(t)$, the
difference of the Maslov indexes with respect to the 
two Lagrangian subspaces is 
given in terms of the H\"{o}rmander index, i.e.   \cite{RS1} (Th.3.5.) \bea
s(V_0, V_1; \Lambda(0), \Lambda(1)) = \mu(V_0,
\Lambda)-\mu(V_1,\Lambda).\eea Obviously, \bea s(V_0, V_1;
\Lambda(0), \Lambda(1)) = s(V_0, V_1; e^{-\varepsilon J}\Lambda(0),
e^{-\varepsilon J}\Lambda(1)), \lb{hp}\eea for $\varepsilon>0$ small
enough.
 The H\"{o}rmander index is
independent of the choice of the path connecting $\Lambda(0)$ and $\Lambda(1)$.
Under the  non-degenerate condition, i.e.  $V_1,V_2$  are
transversal to $\Lambda(0), \Lambda(1)$ correspondingly. Two basic  properties are given below
\bea s(V_0,V_1;\Lambda(0),\Lambda(1))=-s(V_1,V_0;
\Lambda(0),\Lambda(1) ),\nonumber\eea\bea
s(\Lambda(0),\Lambda(1);V_0,V_1)=-s(V_0,V_1; \Lambda(0),\Lambda(1)
),\nonumber\eea If $V_j = Gr(A_j)$, $\Lambda(j)=Gr(B_j)$ for symmetry
matrices $A_j$ and $B_j$, then \bea s(V_0,V_1;
\Lambda(0),\Lambda(1))=\frac{1}{2}sign(B_0-A_1)+\frac{1}{2}sign(B_1-A_0)-
\frac{1}{2}sign(B_1-A_1)-\frac{1}{2}sign(B_0-A_0),
\lb{hc} \eea
where for a symmetric  matrix $A$,   $sign(A)$  is  the signature of the symmetric form 
$\langle A\cdot, \cdot\rangle$. 
 A direct corollary shows that \bea |s(V_0,V_1;
\Lambda(0),\Lambda(1))|\leq 2n.  \lb{hormd}\eea A sharp estimate for  the difference  
of Neumann and Dirichlet boundary conditions has been given in
\cite{LZZ}.

Let $\ga$ be a
fundamental solution of a periodic orbit, then $\ga\in C([0,T],\Sp(2n))$ with $\ga(0)=I_{2n}$,  
as we have mentioned in the
introduction, the Maslov index $\mu(V_n,\ga V_n)$, $\mu(V_d,\ga
V_d)$ and Maslov-type index $i_1(\ga)$, $i_{-1}(\ga)$ play an
important role in the study of stability problem.

We come back to ERE, and we recall that
$\hat{\ga}_e(\tau)=S(\tau)\ga_e(\tau)S^{-1}(0)$, with
$S=diag\(q^{\frac{1}{2}}I_k,q^{-\frac{1}{2}}I_k\)\in\Sp(2k)$. Please note that 
 $S(\mathcal{T})=S(0)$, and the path $S(\tau)$ is contractible in $\Sp(2n)$.  In fact, if we set $S_\alpha=diag\(q_\alpha^{\frac{1}{2}}I_k,q_\alpha^{-\frac{1}{2}}I_k\)$ with $q_\alpha=(1+\alpha\cos(t))^{1/2}$,  then $S_\alpha$ is homotopy to the constant path $S_0(\tau)\equiv I_{2k}$ by $S_\alpha$ for $\alpha\in[0,e]$.  We have
\begin{lem}\lb{lem2.3}  For $e\in[0,1)$,   suppose  $\dim Gr(S_\alpha(\mathcal{T})\ga_{e}(\mathcal{T}) S_\alpha^{-1}(0))\cap\Lambda$ is constant for $\alpha\in[0,e]$, then    $\mu(Gr(\ga_{e}),\Lambda)=\mu(Gr(\hat{\ga}_{e}),\Lambda)$.
\end{lem}
From Lemma \ref{lem2.2} and Lemma \ref{lem2.3}, we get
\bea  i_\omega(\ga_e)=i_\omega(\hat{\ga}_e), \,\ \forall \omega\in\U,\,\  e\in[0,1).\lb{equindex}  \eea

  We first consider the Maslov index on $l^-_0$.  Let
 $\Psi_{0}(\tau)$ be the fundamental solution on $l_0$, that is  \bea
\frac{d}{d\tau}\Psi_{0}(\tau)=J\hat{B}_{0}\Psi_{0}(\tau), \quad
\Psi_{0}(0)=I_{2k},  \quad\tau\in[-\infty,+\infty),  \lb{l0} \eea
with $\hat{B}_{0}= \left( \begin{array}{cccc} I_{k} & \frac{Q_{l_{0}}}{4}I_k \\
\frac{Q_{l_{0}}}{4}I_k & -R \end{array}\right)$.
\begin{prop}\lb{lem2.5}
Suppose $\lambda_1(R)=-(1/8+r_1)$ with  $r_1>0$, then, we  have \bea \mu(V_{d},\Psi_{0}(\tau)V_{d}, 
\tau\in [0, \tau_{0}])\geq\[\frac{\sqrt{r_1}}{\pi}\tau_0\], \eea
where $[Z]$ denote the maximum integer which is not bigger than $Z$.
\end{prop}

\begin{proof}
By changing the basis, we assume $R=diag(\lambda_{1},...,\lambda_{k})$
where $\lambda_{1}\leq\lambda_{2}\leq\ldots\leq\lambda_{k}$ and
$\lambda_{1}<-\frac{1}{8}$.
 Based on the property V of the Maslov index, we have the decomposition
\bea
\mu(V_{d},\Psi_{0}(\tau)V_{d})=\sum_{i=1}^{k}\mu(V_{d}^{1},\Psi_{0}^{i}(\tau)V_{d}^{1}), \nonumber
\eea
where  $\Psi_{0}^{i}(\tau)$ satisfies the equation
\bea \frac{d}{d\tau}\Psi_{0}^{i}(\tau)=J_{2}\hat{B}_{i}\Psi_{0}^{i}(\tau), \quad \Psi_{0}^{i}(0)=I_{2},
\quad\tau\in[0,+\infty),  \lb{l00} \eea
with $ \hat{B}_{i}=\left( \begin{array}{cccc} 1 & \frac{Q_{l_{0}}}{4} \\
\frac{Q_{l_{0}}}{4} & -\lambda_{i}(R) \end{array}\right)$.
Since $\hat{B}_{i}|_{V^{1}_{d}}>0$, then $\Gamma(\Psi_{0}^{i}(\tau)V^{1}_{d}, V^{1}_{d}, \tau)>0$, 
this implies that
$\mu(V_{d}^{1},\Psi_{0}^{i}(\tau)V_{d}^{1}, \tau\in [0, +\tau_0])$ is nondecreasing with respect to $\tau_0$.
Moreover we have
\bea\mu(V_{d}^{1},\Psi_{0}^{i}(\tau)V_{d}^{1}, \tau\in [0, \tau_0])=\sum_{0<\tau_{j}<\tau_0}\nu^{i}(\tau_{j}),
\nonumber\eea
where $\nu^{i}(\tau_{j})= \dim V^{1}_{d}\cap \Psi^{i}_{0}(\tau_{j})V^{1}_{d}$. 
In order to compute the Maslov index
$\mu(V_{d}^{1},\Psi_{0}^{i}(\tau)V_{d}^{1}, \tau\in [0, \tau_0])$, we choose the basis $e_1=(1,0)^{T}$ 
of $V^{1}_{d}$ 
and we let $e^{i}_{1}(\tau)=\Psi^{i}_{0}(\tau)e_{1}$. Then $\mu(V_{d}^{1},\Psi_{0}^{i}
(\tau)V_{d}^{1}, \tau\in [0, \tau_0])$ is equals to the number of zeros
of $f^{i}(\tau)=\det(M^{i}(\tau))$, for  $M^{i}(\tau)=(e_{1},e^{i}_{1}(\tau))$.

Let $\Psi_{0}^{i}(\tau)=\left(
                              \begin{array}{cc}
                                a_{i}(\tau) & b_{i}(\tau) \\
                                c_{i}(\tau) & d_{i}(\tau) \\
                              \end{array}
                            \right)
$. Then $f^{i}(\tau)=c_{i}(\tau)$. From equation (\ref{l00}), we get that $c_{i}(\tau)$ satisfies the equation
\bea \frac{d^{2}}{d\tau^{2}}c_{i}(\tau)&=&\left(\frac{3}{8}\tanh^{2}(\frac{\sqrt{2}}{2}\tau)-\frac{1}{4}+
\lambda_{i}(R)\right)c_{i}(\tau), \nonumber \\
c_{i}(0)&=&0, \,\
\dot{c}_{i}(0)=1. \nonumber \eea
For $i=1$, $\lambda_{1}(R)=-\frac{1}{8}-r_1$, then we have
$\frac{3}{8}\tanh^{2}(\frac{\sqrt{2}}{2}\tau)-\frac{1}{4}+\lambda_{1}(R)\leq-r_1$. Using the Sturm comparison
theorem, we know
the number of zeros of $c_{i}(\tau)$ will be no less than $\[\frac{\sqrt{r_1}}{\pi}\tau_0\]$.
This is complete the proof. 
\end{proof}

In order to proof the Theorem \ref{thm1.1}, we need the lemma
below. We will give an estimation of Maslov  index on the period $[\tau_1,\tau_2]$.
Let $\varepsilon_1=\min \{\frac{r_1}{2r_1+5}, 1/8  \}$,  and $\hat{\ga}(\tau, \tau_1)$ be the 
fundamental solution of (\ref{rg})
with $\hat{\ga}(\tau_{1},\tau_{1})=I_{2k}$,  we have
\begin{lem}\lb{lem2.6}
For  $\varepsilon\leq \frac{1}{2}\varepsilon_1$, $\hat{e}<\varepsilon^3$, \bea
\mu(V_d, \hat{\ga}(\tau,\tau_1)V_d;\tau\in[\tau_1,\tau_2])
\geq \frac{\sqrt{r_1}}{\pi}\ln\(\frac{\varepsilon^2}{\sqrt{\hat{e}}}\)-3. \lb{z.01}\eea
\end{lem}
\begin{proof}   Let $\hat{B}_-= \left( \begin{array}{cccc} I_{k} & -\frac{\sqrt{2}}{4}I_k\\
-\frac{\sqrt{2}}{4}I_k & -R \end{array}\right)$. Suppose $ |\frac{Q+\sqrt{2}}{4}-q|<\varepsilon_1$ for some
 $\varepsilon_1<1$, then we have
  \bea \hat{B}-\hat{B}_-= \left( \begin{array}{cccc} 0_{k} & \frac{Q+\sqrt{2}}{4}I_k-q\mathbb{J}_{k/2} \\
\frac{Q+\sqrt{2}}{4}I_k+q\mathbb{J}_{k/2} & q^2 \end{array}\right)> -\varepsilon_1I_{2k}. \nonumber
\eea
Let $\hat{B}_{\varepsilon_1}=\hat{B}_--{\varepsilon_1} I_{2k}$, by the monotonicity 
property of Maslov index, we have
\bea \mu(V_d, \hat{\ga}(\tau,\tau_1)V_d;\tau\in[\tau_1,\tau_2])\geq \mu(V_d, \exp((\tau-\tau_1)
J\hat{B}_\varepsilon) V_d; \tau\in[\tau_1,\tau_2]).  \lb{z.02}   \eea

Since $R$ is diagonalizable, by using the $\diamond$-product, we can split  
$ \hat{B}_{\varepsilon_1}$ into the product of $k$ two by two matrices, where the 
first factor (which is needed for computing the Maslov index) is given by
$\hat{B}_1:=\left( \begin{array}{cccc} 1-\varepsilon_1 & -\frac{\sqrt{2}}{4}\\
-\frac{\sqrt{2}}{4} &  1/8+r_1-\varepsilon_1 \end{array}\right) $. 
By the direct sum property  of Maslov index, we have
\bea  \mu(V_d, \exp((\tau-\tau_1)J\hat{B}_{\varepsilon_1}) V_d; \tau\in[\tau_1,\tau_2])\geq \mu(V^1_d, 
\exp((\tau-\tau_1)J\hat{B}_1) V^1_d; \tau\in[\tau_1,\tau_2]). \lb{z.03} \eea
Let $f(\varepsilon_1,r_1)=\varepsilon_1^2-(r+\frac{9}{8})\varepsilon_1+r_1$,
$\tilde{B}=diag(1,f(\varepsilon_1,r_1))$, $P=\left( \begin{array}{cccc} (1-\varepsilon_1)^{-1/2} & 0\\
0 & (1-\varepsilon_1)^{1/2} \end{array}\right) \left( \begin{array}{cccc} 1 & \frac{\sqrt{2}}{4}\\
0 & 1 \end{array}\right)$, then
\bea P^{-1} \exp((\tau-\tau_1)J\hat{B}_1)P=\exp((\tau-\tau_1)J\tilde{B}).  \nonumber  \eea
We have \bea \mu(V_d, \exp((\tau-\tau_1)J\hat{B}_1V_d )&=&\mu(P^{-1}V_d, \exp((\tau-\tau_1)J
\tilde{B}P^{-1}V_d ) \nonumber \\ &\geq& \mu(V_d, \exp((\tau-\tau_1)J\tilde{B}V_d )-2, \lb{z.04}  \eea
where the last inequality is from (\ref{hormd}).
In the next, we will estimate  $\mu(V_d, \exp((\tau-\tau_1)J\tilde{B}) V_d; \tau\in[\tau_1,\tau_2])$. 
A direct computation
shows that  $f(\varepsilon_1,r_1)>r_1/2$ if $\varepsilon_1<\frac{r_1}{2r_1+5}$, 
and hence $\tilde{B}>\tilde{B}_{r_1/2}:=diag(1,r_1/2) $.
Moreover 
\bea \mu(V_d, \exp((\tau-\tau_1)J\tilde{B}_{r/2}) V_d; \tau\in[\tau_1,\tau_2])\geq 
\frac{\sqrt{r_1}(\tau_2-\tau_1)}{\sqrt{2}\pi}-1. \lb{z.04a} \eea
 For $\varepsilon<\frac{\varepsilon_1}{2}$, then $ |\frac{Q+\sqrt{2}}{4}-q|<\varepsilon_1$,   
 From (\ref{z.04}, \ref{z.04a}) and  (\ref{adde1}),  we have
\bea \mu(V_d, \ga(\tau,\tau_1)V_d;\tau\in[\tau_1,\tau_2])&\geq& \mu(V_d, \exp((\tau-\tau_1)J\hat{B}_1V_d; 
\tau\in[\tau_1,\tau_2] )\nonumber \\ &\geq&  \mu(V_d, 
\exp((\tau-\tau_1)J\tilde{B}_{r/2}) V_d; \tau\in[\tau_1,\tau_2])-2\nonumber \\ &\geq& \frac{\sqrt{r_1}
\ln(\frac{\varepsilon}{q(\tau_1)})}{\pi}-3 . \lb{z.05} \eea
Direct compute show that $q^2(\tau_1)=\frac{1}{2}(\varepsilon(2\sqrt{2}-\varepsilon)-((\varepsilon(2\sqrt{2}-
\varepsilon))^2-8\hat{e})^{\frac{1}{2}})\leq \frac{8\hat{e}}{3\varepsilon(2\sqrt{2}-\varepsilon)}$, then
\bea  \ln\(\frac{\varepsilon}{q(\tau_1)}\)=\ln(\varepsilon)-\frac{1}{2}\ln(q^2(\tau_1))
\geq\ln\(\frac{\varepsilon^2}{\sqrt{\hat{e}}}\).  \lb{z.06}\eea
The result is from (\ref{z.05}-\ref{z.06}).
\end{proof}

{\bf Proof of Theorem \ref{thm1.1}.}
Under the assumption  $\lambda_1(R)=-\frac{1}{8}-r_1$, from Lemma \ref{lem2.6},  we have
For  $\varepsilon\leq \frac{1}{2}\varepsilon_1$, $\hat{e}<\varepsilon^3$,  $ \mu(V_d, \ga(\tau,\tau_1)
V_d;\tau\in[\tau_1,\tau_2])
\geq \frac{\sqrt{r_1}}{\pi}\ln\(\frac{\varepsilon^2}{\sqrt{\hat{e}}}\)-3$.
Similarly
\bea \mu(V_d, \ga(\tau,\tau_4)V_d;\tau\in[\tau_4,\tau_5])
\geq \frac{\sqrt{r_1}}{\pi}\ln\(\frac{\varepsilon^2}{\sqrt{\hat{e}}}\)-3. \eea
We have \bea \mu(V_d, \ga(\tau,0)V_d;\tau\in[0,\mathcal{T}])
\geq 2\frac{\sqrt{r_1}}{\pi}\ln\(\frac{\varepsilon^2}{\sqrt{\hat{e}}}\)-6. \eea
The results now follows from the fact that $i_1(\ga)\geq \mu(V_d,\ga V_d)$ (for 
Lagrangian system, we refer \S 5.1 for a detailed discussion).  This complete the proof.   $\square$

\section{Collision index for Planar central configurations}
This section is the main part of our paper.  We give the definition
of the collision index in \S 3.1 and we prove the approximation
theorem; we study the basic property of collision index
and we compute in detail the collision index on $l_0$ in \S 3.2.

\subsection{Collision index }

In this subsection, we will consider the Maslov index on the half line
with a hyperbolic equilibrium. This is similar with the case of
homoclinic orbit \cite{CHu} and heteroclinic orbit \cite{HP}, and a
detailed study to the half-clinic orbits is given in \cite{BHPT, HP}.

To define the Maslov index of  the half line, we firstly review some
basic fact of heteroclinic orbits.  We consider the Hamiltonian flow
induced by
\begin{eqnarray}
\dot{z}=JB(t)z,\,\  t\in {\mathbb  R}.  \lb{hr} \end{eqnarray} We assume
the limit is hyperbolic, meaning that \bea  JB(\pm\infty)= \lim_{t\to
\infty} JB(\pm t)\nonumber \eea is hyperbolic. It follows  that ${\mathbb 
R}^{2n}=V^\pm_s\oplus V^\pm_u$, where $V^\pm_s$($V^\pm_u$) is the
stable subspace(resp. unstable subspace) of the equilibria which is
spanned by the generalized eigenvector of eigenvalue with negative
real part (positive real part) of $JB(\pm\infty)$. Moreover,
 both the stable subspace $V^\pm_s$ and the unstable subspace $V^\pm_u$
are Lagrangian subspaces of $({\mathbb  R}^{2n},\omega_0)$. The topology
of Lagrangian Grassmannian $Lag(2n)$ is given by the metric \bea
\rho(V,W)=\parallel \mathcal{P}_V-\mathcal{P}_W\parallel,  \nonumber\eea where
$\mathcal{P}_V, \mathcal{P}_W$ is the orthogonal projection to $V,W$
and $\parallel.\parallel$ is the operator norm.

   Let
$\ga(t,\nu)$ satisfy (\ref{hr}) with $\ga(\nu,\nu)=I_{2n}$.  
In what follows we set $\ga(t):=\ga(t,0)$.   Clearly
$\ga$ satisfies a semigroup property; that is,
$\ga(t,\nu)\ga(\nu,\tau)=\ga(t,\tau)$. For $\nu\in{\mathbb  R}$,
define
\begin{eqnarray}
V_s(\nu)=\{\xi|\xi\in{\mathbb  R}^{2n}\mbox{ and }
\lim_{t\to\infty}\ga(t,\nu)\xi=0\},   \nonumber
\end{eqnarray}
and
\begin{eqnarray}
V_u(\nu)=\{\xi|\xi\in{\mathbb  R}^{2n}\mbox{ and } \lim_{t\to
-\infty}\ga(t,\nu)\xi=0\}.   \nonumber
\end{eqnarray}
We remark that
\begin{eqnarray}
\lim_{\nu\to\infty}V_s(\nu)=V^+_s \mbox{ and
}\lim_{\nu\to-\infty}V_u(\nu)=V^-_u.   \nonumber
\end{eqnarray}
It is well known that both $V_s(\nu)$ and $V_u(\nu)$ are Lagrangian
subspaces of $({\mathbb  R}^{2n},\omega_0)$. An important property from
\cite{AM} isthe following: if $V$ transversal to $V_s(0)$, then \bea
\lim_{t\to\infty}\ga(t,0)V=V^+_u.   \nonumber\eea Similarly, if  $V$
transversal to $V_u(0)$, then \bea \lim_{t\to
-\infty}\ga(t,0)V=V^-_s.  \nonumber\eea

Let ${\mathbb  R}^\pm:=\{\pm x\geq0, x\in {\mathbb  R} \}$. We will define
the Maslov index of  the half line ${\mathbb  R}^+$  or $\R^-$.
We notice that the discussions for heteroclinic orbit  works for the
half-clinic orbit. Firstly, we give the definition of nondegeneracy.
\begin{defi} \lb{dfnd} i) The linear system (\ref{hr}) on $\mathbb{R}$
is called nondegenerate if there is no bounded solution, \\
ii) the linear system on $\mathbb{R}^\pm$ is called nondegenerate
with respect to $V_0$, if there is no bounded solutions on
$\mathbb{R}^\pm$ which satisfies $z(0)\in V_0$.

\end{defi}
We observe that for the system with hyperbolic
limit, all the bounded solution must decay to $0$ as
$t\rightarrow\pm\infty$  \cite{AM}.

  We firstly give the definition of Maslov index on $\R^+$. For, let
$V_0,V_1\in Lag(2n)$ and we suppose that the system is nondegenerate with
respect to $V_0$, that is  $V_0\pitchfork V_s(0)$.   Then
$\ga(t,0)V_0$ is a path of Lagrangian subspaces having  limit $V^+_u$ and 
 so, we define the Maslov index on $\R^+$ with $V_0,V_1$ by \bea
i_+(V_1,V_0):=\mu(V_1,\ga(t,0)V_0, t\in\R^+). \lb{cl0}\eea In the case
$\R^-$, recall that $V_u(t)$ is a path of Lagrangian subspace and
$\lim_{t\to-\infty}V_u(t)=V^-_u$, then for $V\in Lag(2n)$, we define
\bea i_-(V):=\mu(V,V_u(t), t\in\R^-). \lb{cl1}\eea We observe that the
definition on $\mathbb{R}^-$ does not need the nondegenerate
condition. Finally, we will define the Maslov index on $\R$ which is
fully studied in \cite{HP}. Supposing that the linear system is
nondegenerate on $\R$, then $\lim_{t\to-\infty}V_u(t)=V^-_u$ and
$\lim_{t\to\infty}V_u(t)=V^+_u$. Thus we define \bea i(V):=\mu(V,V_u(t),
t\in\R).\lb{clr} \eea Under the nondegenerate condition, it is obvious
that \bea i(V)=i_-(V)+i_+(V,V_u(0)).  \nonumber\eea

We come back to ERE. By assuming that $\lambda_1(R)>-\frac{1}{8}$, we can
identify $l_0$, $l_+$ with $\R$, and identify  $l^\mp_0$, $l^\pm_+$
with $\R^\pm$. For $V_0,V$ satisfying  the nondegenerate conditions,
$i(V_1)$ on $l_0$ or $l_+$, $i_+(V_1,V_0)$ on $l_0^-$ or $l_+^+$ and
$i_-(V_1)$ on $l_0^+$ and on $l_+^-$ are well defined, and we shall refer to 
as {\it collision index}.
\begin{defi}\label{defcn} The planar central configuration is called collision
nondegenerate if the corresponding system on $l_+$ is
nondegenerate.
\end{defi}
We identify $\R$ with $l_+$,  and  let $V_u(\tau)$ be the unstable
subspace. Under the nondegenerate conditions, \bea
\lim_{\tau\rightarrow\pm\infty}V_u(\tau)=V_u^\pm.  \nonumber\eea Let $V_{u,0}$
be the unstable subspace  on $l_0^-$, then \bea
\lim_{\tau\rightarrow\pm\infty}V_{u,0}(\tau)=V_u^\mp. \nonumber  \eea

For   $V_0, V_1\in
Lag(2n)$, satisfying $V_0\pitchfork V_s(0)$,   then the Maslov index
$ i_+(V_0,V_1)$  and  $ i_-(V_1)$ are well defined. As $e\to1$, we
have the next approximation theorem which plays a key role in our
paper.

\begin{thm}\lb{th.pr} Assuming $\lambda_1(R)>-\frac{1}{8}$, we have: 
(i) If $V^-_u\pitchfork V_1$, the system is nondegenerate with respect to $V_0$ on $l_0^-$,
and nondegenerate with respect to $V_1$ on  $l^-_+$, then,  for $1-e$ small
enough, $V_1\pitchfork \hat{\ga}_e(\mathcal{T}/2) V_0 $ and
 \bea \mu(V_1,\hat{\ga}_e(\tau) V_0, \tau\in[0,\mathcal{T}/2] )=
i_+(V_1,V_0; l_0^- )+i_-(V_1;l_+^-).\eea
(ii) If $V^+_u\pitchfork V_1$, the system is nondegenerate with
respect to $V_0$ on $l_+^+$, and nondegenerate with respect to $V_1$
on $l_0^+$,  then,  for $1-e$ small
enough, $V_1\pitchfork
\ga_e(\mathcal{T})\hat{\ga}^{-1}_e(\mathcal{T}/2) V_0 $ and  \bea 
\mu(V_1,\hat{\ga}_e(\tau)\hat{\ga}^{-1}_e(\mathcal{T}/2) V_0,
\tau\in[\mathcal{T}/2,\mathcal{T}] )= i_+(V_1,V_0; l_+^+
)+i_-(V_1;l_0^+).\eea (iii) If
  $V^\pm_u\pitchfork V_1$, the system is
collision nondegenerate,  and  nondegenerate with respect to $V_0$, $V_1$ on
$l_0^-$,  $l_0^+$ correspondingly,  then,  for $1-e$ small
enough, $V_1\pitchfork \hat{\ga}_e(\mathcal{T})  V_0 $ and \bea
\mu(V_1,\hat{\ga}_e(\tau) V_0, \tau\in[0,\mathcal{T}] )= i_+(V_1,V_0;
l_0^-)+i_-(V_1;l_0^+)+i(V_1;l_+). \lb{pr1} \eea
\end{thm}

The proof is based on a series lemmas, we firstly give the
 next lemma  which is from \cite{MSS2}.
\begin{lem}\lb{addle1}  Let us consider 
the linear system \bea x'=Dx+C(\tau)x, \lb{adle1.1} 
\eea where $D$ is $k\times k$ diagonal matrix and $C(t)$ is a 
continuous matrix in $t\in[0,\hat{t}]$, such that $\int_0^{\hat{t}}\|C(s)\|ds<\hat{\varepsilon}$, 
for some constant $\hat{\varepsilon}$ which 
satisfies  $\frac{6\sqrt{k}\hat{\varepsilon}}{1-3\hat{\varepsilon}}<1$ and let $\ga(\tau)$ be the 
fundamental solution of (\ref{adle1.1}). Then 
for $t\in[0,\hat{t}]$, we have
\bea \ga(t)=(I+O(t))\exp(Dt)(I+\cal{S}), \eea where $\|O(t)\|\leq \frac{3\sqrt{k}\hat{\varepsilon}}
{1-3\hat{\varepsilon}}$, 
$\|\cal{S}\|\leq \frac{6\sqrt{k}\hat{\varepsilon}}{1-3\hat{\varepsilon}}$.
\end{lem}
\begin{proof} From lemma 6 of \cite{MSS2} and for $\hat{\varepsilon}<1/4$, let $\lambda$ be an 
eigenvalue of $D$ and $V$ be the corresponding
eigenvector.Then  there exists a solution  $\varphi(t)$ of $(\ref{adle1.1})$ such that
$$ \|e^{-\lambda t}\varphi(t)-V\|\leq \frac{3\hat{\varepsilon}}{1-3\hat{\varepsilon}}.$$
Let $e_i, i=1,\cdots,k$ be the canonical basis, $Y(t)$ be the matrix defined by $\varphi_1,\cdots,\varphi_k$
as column vectors.  
We define $O(t):=Y(t)\exp(-Dt)-I$, then $\|O(t)\|\leq \frac{3\sqrt{k}\hat{\varepsilon}}{1-3\hat{\varepsilon}}$ 
for $t\in[0,\hat{t}]$.  
Obviously, $\ga(t)=Y(t)Y^{-1}(0)=(I+O(t))\exp(Dt)(I+O(0))^{-1}$. Let $\cal{S}:=(I+O(0))^{-1}-I$, 
then for $\|O(0)\|<1/2$, we have 
$$\|\cal{S}\|\leq \frac{\|O(0)\|}{1-\|O(0)\|}<2\|O(0)\|\leq \frac{6\sqrt{k}\hat{\varepsilon}}
{1-3\hat{\varepsilon}}, $$
which complete the proof.
\end{proof}

Assume that $\lambda_{1}(R)>-\frac{1}{8}$, then $P_{\pm}$ are
hyperbolic. We recall that we  set  $D_\pm=J\hat{B}(P_\pm)$  
having  the form   $J_k\left( \begin{array}{cccc} I_{k} & \pm\frac{\sqrt{2}}{4}I_k \\
\pm\frac{\sqrt{2}}{4}I_k & -R
\end{array}\right) $. An easy computation shows that  
the eigenvalues of $D_\pm$ are  real if $\lambda_{1}(R)>-\frac{1}{8}$. 
Choose basis such that $R$ is diagonalisable, that is $R=diag(\lambda_1,\cdots,\lambda_k)$ 
with $\lambda_1\leq\cdots\leq\lambda_k$. Let $P_1=\left(
          \begin{array}{cc}
            I_{k} & \frac{\sqrt{2}}{4} \\
            0_{k} & I_{k} \\
          \end{array}        \right)$,  $P_2=\left(
          \begin{array}{cc}
            \sqrt{1/8+R}I_{k} &  \sqrt{1/8+R}I_{k}  \\
            I_{k} & -I_{k} \\
          \end{array}        \right)$,  $P=P_1P_2$. By a direct computation we get  that
   \bea  P^{-1}D_-P=diag(\eta_{1},..,\eta_{k},-\eta_{1},...,-\eta_{k}),     \eea
where $\eta_j=\sqrt{1/8+\lambda_j}$.  Let $\hat{\eta}=\max\{\eta_j, \eta_j^{-1}; j=1,\cdots,k\}$, then
easy computation show that \bea  \max\{\|P\|, \|P^{-1}\|\}\leq 2(1+\hat{\eta}). \lb{z.31}  \eea

Based on Lemma \ref{addle1}, we firstly prove the important lemma below.
\begin{lem}
We assume that $\lambda_{1}(R)>-\frac{1}{8}$ and we let $\hat{\ga}(\tau, \tau_1)$ be the 
fundamental solution of (\ref{rg})
with $\hat{\ga}(\tau_{1},\tau_{1})=I_{2k}$. Then for $\varepsilon<\varepsilon_0$, $\hat{e}<\varepsilon^3$,
we have the following estimate below \bea
\hat{\gamma}(\tau,\tau_{1})=P(I_{2k}+\Delta(\tau))D(\tau)(I_{2k}+\cal{S})P^{-1},
\quad  \tau\in[\tau_{1},
\tau_{2}], \lb{im}\eea where the matrices $\Delta(\tau)$, $\cal{S}$ satisfy
$\|\Delta\|\leq \frac{c_1}{2}\varepsilon$, $\|\cal{S}\|\leq c_1\varepsilon$, for 
$\varepsilon_0, c_1$ is constant and dependent on $R$  and
$D(\tau)=diag(e^{\eta_{1}(\tau-\tau_{1})},...,e^{\eta_{k}(\tau-\tau_{1})},
e^{-\eta_{1}(\tau-\tau_{1})},...,e^{-\eta_{k}(\tau-\tau_{1})})$.
\lb{lem3.6}
\end{lem}
\begin{proof}
Simple computation shows that
$J\hat{B}(\tau)=D_-+C(\tau)$ with $$
C(\tau)=\left(
          \begin{array}{cc}
            -\frac{Q+\sqrt{2}}{4}I_{k}-q\mathbb{J}_{k/2} & -q^{2}I_{k} \\
            0_{k} & \frac{Q+\sqrt{2}}{4}I_{k}-q\mathbb{J}_{k/2} \\
          \end{array}
        \right).$$ Let $W(\tau)=P^{-1}\hat{\ga}(\tau,\tau_{1})P$,  then \bea
\dot{W}(\tau)=(P^{-1}D_-P+P^{-1}C(\tau)P)W(\tau). \nonumber \eea
Let $\varepsilon<\varepsilon_0$, where $\varepsilon_0<1/8$ will be fixed later.  It is obvious 
$\int_{\tau_{1}}^{\tau_{2}}q^{2}(\tau)\leq\int_{\tau_{1}}^{\tau_{2}}q(\tau)d\tau$ for $0<q<1$. From
 equations (\ref{adde2}-\ref{adde3}), we have $\int_{\tau_{1}}^{\tau_{2}}q(\tau)\leq 2 \varepsilon$ as 
 well as 
    $\int_{\tau_{1}}^{\tau_{2}}|Q(\tau)+\sqrt{2}|\leq 2\varepsilon$. Then we have
\bea \int_{\tau_1}^{\tau_2}\|C(\tau)\|d\tau\leq \int_{\tau_1}^{\tau_2}(\frac{1}{2}\|Q
+\sqrt{2}\|+2\|q\|+\|q^2\|)d\tau\leq6\varepsilon, \eea
and from (\ref{z.31}) also that
\bea \int_{\tau_1}^{\tau_2}\|P^{-1}C(\tau)P\|d\tau\leq24(1+\hat{\eta})^2 \varepsilon. \eea
Let $\varepsilon_0=(24(3+6\sqrt{k})(1+\hat{\eta})^2)^{-1}$ and  for $\varepsilon<\varepsilon_0$, 
we denote $\hat{\varepsilon}=24(1+\hat{\eta})^2 \varepsilon$. 
Then, $\varepsilon<\varepsilon_0$ implies  $\frac{6\sqrt{k}\hat{\varepsilon}}{1-3\hat{\varepsilon}}<1$.  
From Lemma \ref{addle1}, we have
\bea \|\Delta\|\leq \frac{3\sqrt{k}\hat{\varepsilon}}{1-3\hat{\varepsilon}}\leq \frac{c_1}{2}\varepsilon, 
\quad  \|\cal{S}\|\leq
\frac{6\sqrt{k}\hat{\varepsilon}}{1-3\hat{\varepsilon}}\leq c_1\varepsilon,  \eea
where $c_1=24^2\sqrt{k}(1+\hat{\eta})^2$ only depend on $R$. This complete the proof.
\end{proof}
Given two subspaces, graphs of two linear operators, it is possible to introduce a norm 
topology on the Lagrangian Grassmannian, equivalent to the gap topology, as below. 
More
precisely,  if $E=E^-\oplus E^+$, a sequence of operators
$(L_n)_{n \in \mathbb N}\subset L(E^-,E^+)$ converges to $L$ if and only if their graphs
converge to the graph of $L$. Another important property is
that, the image $TV$ of a closed subspace $V$ by an invertible linear
operator $T$, is continuously depending on $(T,V)$.
From Lemma \ref{lem3.6}, $P^{-1}(V_u)=V_d$  and $P^{-1}(V_s)=V_n$.
For $V\in Lag(2n)$ with
$V\pitchfork V^-_s$, then $\exists L_V$ such that \bea
P^{-1}V=Gr(L_V).  \nonumber\eea  We give a equivalent metric \bea
\hat{\rho}(V,W)=\parallel L_V-L_W
\parallel. \nonumber\eea
It is obvious that $L_{V^-_u}=0_k$, since $V^-_u \pitchfork V_1$,  so $\exists \sigma_1>0$ such
that $V\pitchfork V_1$ if $\|L_V\|\leq \sigma_1$. For $\sigma>0$, we
always denote \bea B_{\hat{\rho}}(V,\sigma)=\{W\in Lag(2n),
\hat{\rho}(V,W)<\sigma \}. \nonumber \eea

\begin{lem}
Let $\bar{V}=Gr(L_{V})$, i.e. $\bar{V}=PV$. Then for any $\bar{V}\in U(\sigma)=\{\bar{V}:
\|L_{V}\|<\sigma\}$ with $\sigma<1$ and if $\|\Lambda\|<\sigma/6$ then  $(I+\Lambda)\bar{V}\in
U(2\sigma)$ .\label{lem3.9}
\end{lem}
\begin{proof}
For $\Lambda=\left(\begin{array}{cc}
               \Lambda_{1} & \Lambda_{2} \\
               \Lambda_{3} & \Lambda_{4}
             \end{array}\right)
$, we have  $(I+\Lambda)\left(\begin{array}{c}
                    x \\
                    L_{V}x
                  \end{array}
\right)= \left(\begin{array}{c}
        (I+\Lambda_{1}+\Lambda_{2}L_{V})x \\
        (\Lambda_{3}+(I+\Lambda_{4})L_{V})x
      \end{array}
\right).$ Let $y=(I+\Lambda_{1}+\Lambda_{2}L_{V})x$ and choose
$\|\Lambda\|$ small enough, then we have
$(I+\Lambda)\bar{V}=Gr((\Lambda_{3}+(I+\Lambda_{4})L_{V})(I+\Lambda_{1}+\Lambda_{2}L_{V})^{-1})$.
Since $\|\Lambda\|<\sigma/6$, an easy computation shows that 
\bea \|(\Lambda_{3}+(I+\Lambda_{4})L_{V})(I+\Lambda_{1}+\Lambda_{2}L_{V})^{-1}\|\leq
\frac{\|\Lambda_{2}\|+\|I+\Lambda_{4}\|\|L_{V}\|}{I-\|\Lambda_{1}\|-\|\Lambda_{2}\|\|L_{V}\|}<
\frac{\|\Lambda\|+(1+\|\Lambda\|)\|L_{V}\|}{I-\|\Lambda\|-\|\Lambda\|\|L_{V}\|} <2\sigma\eea.
\end{proof}
\begin{lem} \label{lem5.4}
For any $0<\sigma<1$, we let $\varepsilon_\sigma:=\min\{\varepsilon_0, \frac{\sigma}{24c_1}\}$.
If   $\varepsilon\leq\epsilon_\sigma$,  $V\in  B_{\hat{\rho}}(V^-_u,\sigma/4)$ and  
$\hat{e}<\varepsilon^3$, then for every $\tau\in[\tau_1,\tau_2]$,  we have  $\hat{\ga}_e(\tau,\tau_1)V\in
 B_{\hat{\rho}}(V^-_u,\sigma)$ .
\end{lem}
\begin{proof} From (\ref{im}),
$P^{-1}\hat{\ga}_e(\tau,\tau_1)V=(I_{2k}+\Delta(\tau))D(\tau)(I_{2k}+\cal{S})P^{-1}V$ where $V\in
B_{\hat{\rho}}(V^-_u,\sigma/4)$,  $P^{-1}V\in U(\sigma/4)$. Since $\|\cal{S}\|\leq C_1\varepsilon
\leq \frac{\sigma}{24}$,  then we have
 $(I_{2k}+S)P^{-1}V\in U(\sigma/2)$ by Lemma \ref{lem3.9}.
Obviously $D(\tau)U(\sigma/2)\subset U(\sigma/2)$,  using Lemma
\ref{lem3.9} again, we have
$(I_{2k}+\Delta(\tau))D(\tau)(I_{2k}+\cal{S})P^{-1}V\in U(\sigma)$, which
conclude the proof. 
\end{proof}

Let $\Psi_0$ be  the fundamental
solution on $l_0$ as given  in Equation (\ref{l0}) 
and $\Psi_+(\tau,\nu)$ be the fundamental
solution on $l_+$.  Let
$\sigma<\frac{1}{3}\min\{\rho(V^-_u,V_1),\rho(V^-_u,V^-_s)\}$ small
enough such that $B_{\rho}(V^-_u,3\sigma)\pitchfork V^-_s,V_1$. For
this $\sigma$, $\exists\sigma_1>0$ such that
$B_{\hat{\rho}}(V^-_u,\sigma_1)\subset B_{\rho}(V^-_u,\sigma)$, and
let $\varepsilon_{\sigma_1}$ be the number corresponding to $\sigma_1$ in
Lemma \ref{lem5.4}.

 Choose $\varepsilon<\sigma_1$
small enough such that\bea
\max\{\rho(\Psi_0(\tau_0)V_0,V^-_u,),\rho(V_u(-\tau_{l_{+}}),V^-_u,),\rho(\Psi_+
(-\tau_{l_{+}},0)V_1,V^-_s,)\}<\sigma.\lb{delta}
\eea  From Lemma \ref{lem5.4},  We have
\begin{lem} \label{lem5.5.0} For this fixed $\varepsilon$, $\hat{e}<\varepsilon^3$,
 $\rho(\hat{\ga}_e(\tau)V_0,V^-_u)<\sigma$
for $\tau\in(\tau_1,\tau_2)$.
\end{lem}
Similarly,  by the symmetric   figure 2,  for $\Sigma_4, \Sigma_5$, we have the following result. 
\begin{lem} \label{lem5.5.1} For this fixed
$\varepsilon$, $\hat{e}<\varepsilon^3$, $\rho(\hat{\ga}_e(\tau)
\hat{\ga}^{-1}_e(\mathcal{T}/2)V_1,V^+_u)<\sigma$ for
$\tau\in(\tau^+_4,\tau^+_5)$.
\end{lem}

{\bf Proof of Theorem \ref{th.pr}.} To prove (i), we will compute
the Maslov index $\mu(V_1,\hat{\ga}_e(\tau)V_0;\tau\in[0,\mathcal{T}/2])$
on the three time interval $[0,\tau_1]$, $[\tau_1,\tau_2]$ and
$[\tau_2,\mathcal{T}/2]$. Form Lemma \ref{lem5.5.0}, for $\hat{e}<\varepsilon^3$,
\bea
\mu(V_1,\hat{\ga}_e(\tau)V_0,\tau\in[0,\tau_1))=\mu(V_1,\Psi_0(\tau)V_0,\tau\in[0,\tau_{l_{0}}]
=i(V_1,V_0;l^-_0).\lb{pf1}
\eea Obviously \bea
\mu(V_1,\hat{\ga}_e(\tau)V_0,\tau\in[\tau_1,\tau_2])=0. \lb{pf2}\eea Now
we consider the path on $[\tau_2,\mathcal{T}/2]$. Please note that
for $\hat{e}<\varepsilon^3$, $\rho(\hat{\ga}_e(\tau_2)V_0,V^-_u)<\sigma$ by Lemma
\ref{lem5.5.0} and $\rho(\Psi_+(-\tau_{l_{+}},0)V_1,V^-_s,)<\sigma$
by (\ref{delta}), then $\hat{\ga}_e(\tau_2)V_0\pitchfork
\Psi_+(-\tau_{l_{+}},0)V_1 $, which implies \bea
\Psi_+(0,-\tau_{l_{+}})\hat{\ga}_e(\tau_2)V_0\pitchfork V_1.  \nonumber
\eea Since $\hat{\ga}_e(\tau-\tau_2)$ uniformly converges to
$\Psi_+(\tau,-\tau_{l_{+}})$, we have for $\hat{e}<\varepsilon^3$ small enough, \bea
\mu(V_1,\hat{\ga}_e(\tau)V_0,\tau\in[\tau_2,\mathcal{T}/2])=\mu(V_1,\Psi_+
(\tau,-\tau_{l_{+}})V_0,\tau\in[-\tau_{l_{+}},0])=i_-(V_1;l^-_+).
\lb{pf3}\eea The result of (i) is from (\ref{pf1}), (\ref{pf2}) and
(\ref{pf3}).

The proof of (ii) is based Lemma \ref{lem5.5.1} and is totally analogous.

To prove (iii), we compute  Maslov index
$\mu(V_1,\hat{\ga}_e(\tau)V_0;\tau\in[0,\mathcal{T}])$ on the five time
intervals $[0,\tau_1]$, $[\tau_1,\tau_2]$, $[\tau_2,\tau_4]$,
$[\tau_4,\tau_5]$ and $[\tau_5,\mathcal{T}]$. By assumption of
collision nondegenerate, the system is nondegenerate on $l_+$, that
is $\lim_{\tau\to+\infty}V_u(\tau)=V^+_u$. So for $\sigma$ small
enough,  $V\in B_\sigma(V^-_u)$,  we have
$\lim_{\tau\to+\infty}\Psi(\tau,-\tau_{l_{+}})V=V^+_u$. If we
consider $-\tau_{l_{+}}$ as the starting point, this means the system is
nondegenerate with respect to $V$. By arguing  as in step  (i), we
have for $\hat{e}<\varepsilon^3$ small enough. \bea
\mu(V_1,\hat{\ga}_e(\tau)V_0,\tau\in[\tau_2,\tau_4))=\mu(V_1,\Psi_+
(\tau,-\tau_{l_{+}})(\hat{\ga}_e(\tau_2)V_0),\tau\in[-\tau_{l_{+}},\tau_{l_{+}}]=i(V_1;l_+), \nonumber
\eea \bea \mu(V_1,\hat{\ga}_e(\tau)V_0,\tau\in[\tau_4,\tau_5))=0, \nonumber\eea and
\bea
\mu(V_1,\hat{\ga}_e(\tau)V_0,\tau\in[\tau_5,\mathcal{T}))=i_-(V_1;l^-_0),   \nonumber\eea
with (\ref{pf1}-\ref{pf2}), we get the result. $\square$

\subsection{Some fundamental property of collision index}

We first compute the  collision index on $l_0$.  Recall that  on
line $l_0$,
$ \hat{B}= \left( \begin{array}{cccc} I_{k} & \frac{Q}{4}I_k \\
\frac{Q}{4}I_k & -R \end{array}\right)$. We can choose bases such
that $R=diag(r_1,\cdots,r_k)$,  and  set $\hat{B}_r=\left( \begin{array}{cccc} 1 & \frac{Q}{4} \\
\frac{Q}{4} & -r \end{array}\right)$.
Given any two $2m_k\times 2m_k$ square block  matrices 
$M_k=\left(\begin{array}{cc}A_k&B_k\\
                                C_k&D_k\end{array}\right)$ with $k=1, 2$,
the symplectic sum of $M_1$ and $M_2$ is defined by
\bea M_1\diamond M_2=\left(
  \begin{array}{cccc}
   A_1 &   0 & B_1 &   0\\
                            0   & A_2 &   0 & B_2\\
                           C_1 &   0 & D_1 &   0\\
                           0   & C_2 &   0 & D_2  \\
  \end{array}
\right).\nonumber
\eea It is clear  that,
$\hat{B}=\hat{B}_{r_1}\diamond\cdots\hat{B}_{r_k}$.   We
start by considering the two dimensional case.  The linear systems
$\dot{z}=J\hat{B}_{r}z$ with the form \bea
\dot{y}&=&-\frac{Q}{4}y+rx,\lb{4.1}\\ \dot{x}&=&y+\frac{Q}{4}x,
\lb{4.2} \eea where $z=(y,x)^T$. Assuming  that $r\neq0$, $r>-\frac{1}{8}$
and by  taking derivative with respect to 
$\tau$ on both sides of (\ref{4.1}), we have \bea
\ddot{y}&=&-\frac{\dot{Q}}{4}y-\frac{Q}{4}\dot{y}+r\dot{x} \nonumber \\ &=&
(\frac{Q^2}{16}-\frac{\dot{Q}}{4}+r)y \nonumber
\\&=&(\frac{1}{4}-\frac{1}{8}\tanh^2(\frac{\sqrt{2}\tau}{2})+r)y, \lb{4.3}\eea
where the second equality is from the fact that
$x=\frac{1}{r}(\dot{y}+\frac{Q}{4}y)$ by (\ref{4.1}). Let
\bea
f:=\frac{1}{4}-\frac{1}{8}\tanh^2(\frac{\sqrt{2}\tau}{2})+r=\frac{1}{8}
(1-\tanh^2(\frac{\sqrt{2}\tau}{2}))+(r+\frac{1}{8})>0.
\eea 
\begin{lem}\label{lem4.1}  If $r>-\frac{1}{8}$, $r\neq0$, then i) for any $t_2>t_1$, 
there is no nontrivial solution of (\ref{4.3}) which satisfies boundary
condition $y(t_1)\dot{y}(t_1)\geq y(t_2)\dot{y}(t_2)$; ii) there is
no nontrivial solution satisfying  $y(0)\dot{y}(0)=0$ and
$y\rightarrow0$, $\dot{y}\rightarrow0$ as
$\tau\rightarrow\pm\infty$; iii) there is no nontrivial bounded
solution on $\mathbb{R}$.
 \end{lem}
\begin{proof}  Suppose $y$ is solution of (\ref{4.3}), then
multiply by $y$ and by integrating over the time interval $[t_1,t_2]$ we have \bea
y(t_1)\dot{y}(t_1)-y(t_2)\dot{y}(t_2)+\int_{t_1}^{t_2}(\dot{y}^2+fy^2)d\tau=0.\lb{lem4.1f}
\eea The first conclusion now readily follows. By taking limit of
(\ref{lem4.1f}) could get the second conclusion. The third
conclusion follows from the fact that any bounded solution must
decay exponential fast.
\end{proof}

Recall that  $V_d,V_n$ are
Lagrangian subspaces corresponding to the Dirichlet and Neumann
boundary conditions. In the two dimensional case,
let $e_1=(1,0)^T$, $e_2=(0,1)^T$. Then it is obvious that  $V_d,V_n$
are the  linear spaces spanned by $e_1,e_2$.  An easy computation shows that $z(0)=e_1$ is
equivalent to $y(0)=1$ and $\dot{y}(0)=0$. In the same way,  $z(0)=e_2$
is equivalent to $y(0)=0$ and $\dot{y}(0)=r$.  The second conclusion of
Lemma \ref{lem4.1}. implies that for $r\neq0$, the system is
nondegenerate. However this is not true for $r=0$. In fact, let $r=0$ in
equations (\ref{4.1}-\ref{4.2}). Then there is a nontrivial solution satisfies
$Z(0)=e_2$ and $Z(t)\rightarrow 0$. Thus we have

\begin{lem}\label{lem4.2}  Supposing  $r>-1/8$,  the system on $l_0^-$ is 
nondegenerate with $V_d$,   and it is  nondegenerate with $V_n$ if and only if  $r\neq0$.
 \end{lem}

We firstly consider the Maslov
index on $l_0^+$, so we have the following result. 
\begin{lem}\label{lem4.6} Suppose $r>-1/8$ and $r\neq0$. Then \bea
i_-(V_d;l_0^+)=i_-(V_n;l_0^+)=0.\lb{i0p}\eea
 \end{lem}
\begin{proof} By the definition of Maslov index, we only need to 
show that there is no nontrivial solution. 
If not,  we assume that  $Z(\tau)$ is the solution of
the equations (\ref{4.1}-\ref{4.2}). Then  $Z(\tau)\rightarrow0$ as
$\tau\rightarrow-\infty$ and then $\dot{y}(\tau)y(\tau)\rightarrow0$.
Let $t_1\rightarrow -\infty$ in (\ref{lem4.1f}), we have \bea
-y(t_2)\dot{y}(t_2)+\int_{-\infty}^{t_2}(\dot{y}^2+fy^2)d\tau=0.
\eea Please note that
$y(t_2)\dot{y}(t_2)=-\frac{Q}{4}y^2(t_2)+rx(t_2)y(t_2) $. Then for
$t_2\in\mathbb{R}^-$, $Z(t_2)\in V_d$ or $V_n$ implies
$y(t_2)\dot{y}(t_2)\leq0$, so we get the result.
\end{proof}

We continuous to compute the Maslov index on $l_0^-$.
\begin{cor}\label{cor4.3} Suppose $r>-1/8$ and $r\neq0$,  then \bea i_+(V_n,V_d;l_0^-)=0,\lb{i0nd}\eea and
\bea i_+(V_n,V_n;l_0^-)=\left\{\begin{array}{ll} 1 & \quad
           {\mathrm if}\; r\in(-\frac{1}{8},0),  \\
           \\
 0 & \quad {\mathrm if}\; r>0.\end{array}\right.\lb{i0nn} \eea  \end{cor}
 \begin{proof} The proof follows from equation (\ref{1.3c}).  Please note
that Lemma \ref{lem4.1}. implies that there are no nontrivial
solutions which satisfy  $Z(0)=e_1,e_2$ and $Z(T)\in V_n$ for some
$T>0$.  There is a crossing for $i_+(V_n,V_n;l_0^-)$ at $T=0$. An easy computation 
shows that the crossing 
form $\Gamma(V_n,V_n,0)$ is positive for $r\in(-\frac{1}{8},0)$ and 
negative for $r>0$, which implies the results. \end{proof}

 To compute the collision index $i_+(V_d,V_d;l_0^-)$ and $i_+(V_d,V_n;l_0^-)$,
we will use the H\"{o}rmander index. We observe that, in the point
$(0,Q_-)$, $J\hat{B}_r(-\infty)=\left( \begin{array}{cccc}\frac{\sqrt{2}}{4}
& r\cr 1 &-\frac{\sqrt{2}}{4}\end{array}\right)$ and hence the eigenvalues
$\lambda_\pm=\pm(\frac{1}{8}+r)^{\frac{1}{2}}$ with eigenvector
$e^-_\pm=(\frac{\sqrt{2}}{4}\pm(\frac{1}{8}+r)^{\frac{1}{2}},1)^T$.
The unstable subspace $V^-_u$ is spanned by $e^-_+$ and stable
subspace $V^-_s$ is spanned by $e^-_-$. Similarly, at $(0,Q_+)$,
$J\hat{B}_r(+\infty)=\left( \begin{array}{cccc}\frac{-\sqrt{2}}{4} & r\cr 1
&\frac{\sqrt{2}}{4}\end{array}\right)$, the eigenvalues
$\lambda_\pm=\pm(\frac{1}{8}+r)^{\frac{1}{2}}$ with eigenvector
$e^+_\pm=(-\frac{\sqrt{2}}{4}\pm(\frac{1}{8}+r)^{\frac{1}{2}},1)^T$.
The unstable subspace $V^+_u$ is spanned by $e^+_+$ and stable
subspace $V^+_s$ is spanned by $e^+_-$. The H\"{o}rmander index could be
computed  by (\ref{hp}), (\ref{hc}), or we just
choose simple Lagrangian paths connected $\Lambda(0), \Lambda(1) $ and 
compute the difference. For reader's convenience, we list the result
below. \bea s(V_d,V_n,V_d,V^-_u)&=&1,\lb{s1}\\
s(V_d,V_n,V_n,V^-_u)&=&0,\lb{s2} \eea
 \bea
s(V_n,V_d,V^-_u,V^+_u)=\left\{\begin{array}{ll} 1 & \quad
           {\mathrm if}\; r\in(-\frac{1}{8},0),  \\
           \\
 0, & \quad {\mathrm if}\; r>0,\end{array}\right. \lb{s3} \eea
\bea  s(V_d,V^-_u,V_n,V^-_s)= s(V_d,V^-_u,V_d,V^-_s)=0.\lb{s4} \eea
The following  Figures illustrate the H\"{o}rmander index
(\ref{s1}-\ref{s4}), where $y$ is the
horizontal coordinate, $x$ is the vertical coordinate and the 
anticlockwise rotation is
 positive rotation.
\begin{figure}[H]
\begin{minipage}[t]{0.5\linewidth}
\centering
\includegraphics[height=0.6\textwidth,width=0.78\textwidth]{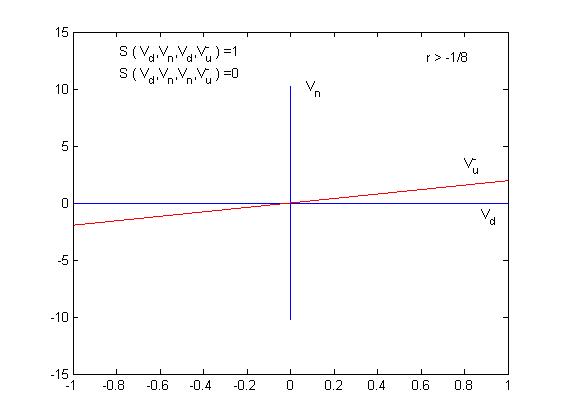}
     \caption{ $s(V_d,V_n,V_d,V^-_u)$ and $s(V_d,V_n,V_n,V^-_u)$}
\label{fig:side:a2}
\end{minipage}%
\begin{minipage}[t]{0.5\linewidth}
\centering
\includegraphics[height=0.6\textwidth,width=0.78\textwidth]{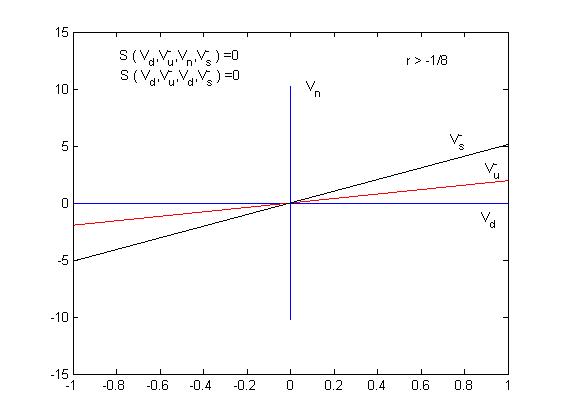}
     \caption{ $s(V_d,V^-_u,V_n,V^-_s)$ and $s(V_d,V^-_u,V_d,V^-_s)$}
\label{fig:side:b2}
\end{minipage}
\end{figure}
\begin{figure}[H]
\begin{minipage}[t]{0.5\linewidth}
\centering
\includegraphics[height=0.6\textwidth,width=0.78\textwidth]{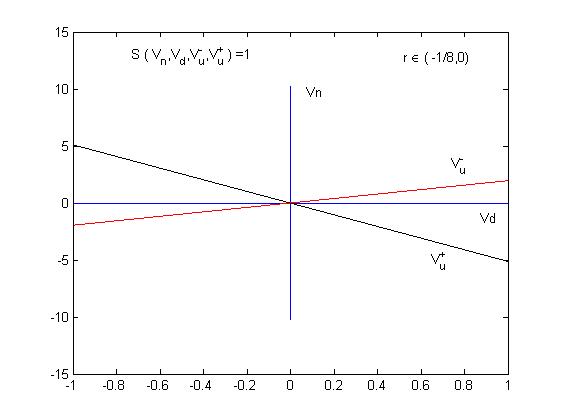}
     \caption{ $s(V_n,V_d,V^-_u,V^+_u)$ for $r\in(-\frac{1}{8},0)$.}
\label{fig:side:a3}
\end{minipage}%
\begin{minipage}[t]{0.5\linewidth}
\centering
\includegraphics[height=0.6\textwidth,width=0.78\textwidth]{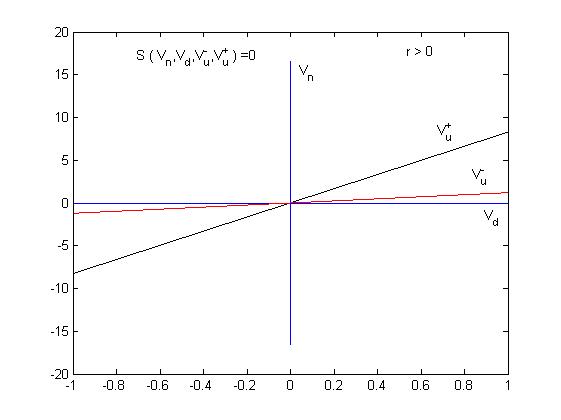}
     \caption{ $s(V_n,V_d,V^-_u,V^+_u)$ for $r>0$.}
\label{fig:side:b3}
\end{minipage}
\end{figure}
From Lemma \ref{lem4.2}., the
system is nondegenerate for $r\neq0$, so
$\ga(\tau)V_d\longrightarrow V^-_u$ and $\ga(\tau)V_n\longrightarrow
V^-_u$ as $\tau\longrightarrow+\infty$. We have \bea
i_+(V_d,V_d;l_0^-)=i(V_n,V_d;l_0^-)+s(V_d,V_n,V_d,V^-_u), \lb{4.8}
\eea
 \bea i_+(V_d,V_n;l_0^-)=i(V_n,V_n;l_0^-)+s(V_d,V_n,V_n,V^-_u). \lb{4.9} \eea
 Thus from equations (\ref{s1}-\ref{s2}) we have
\begin{cor}\label{cor4.4}  Suppose $r>-1/8$ and $r\neq0$, \bea i_+(V_d,V_d;l_0^-)=1,\lb{i0dd}\eea and
\bea i_+(V_d,V_n;l_0^-)=\left\{\begin{array}{ll} 1 & \quad
           {\mathrm if}\; r\in(-\frac{1}{8},0),  \\
           \\
 0, & \quad {\mathrm if}\; r>0.\end{array}\right.\lb{i0dn} \eea  \end{cor}

We come back to the higher dimension.
We denote by $\phi(R)$  the number of  total negative eigenvalues of $R$. It is obvious that, $\phi(R)=\phi(a_0)$ which is the
 Morse index of the central configuration $a_0$.  By property V of Maslov index, we have
\begin{cor}\label{cor4.5} Supposing $\lambda_1(R)>-1/8$, $R$ is nondegenerate, then we have 
\bea i_-(V_d;l^+_0)=0, \quad i_-(V_n;l^+_0)=0,\lb{i01} \eea  \bea i_+(V_d,V_d;l_0^-)=k,\quad i_+(V_d,V_n;l_0^-)=\phi(R),\lb{i02}\eea and
\bea i_+(V_n,V_d;l_0^-)=0,\quad  i_+(V_n,V_n;l_0^-)=\phi(R).
\lb{i03} \eea
\end{cor}

We then consider the linear system on $l_+$, and recall that under the
collision nondegenerate condition $i(V)$ is well defined.
\begin{prop} Under the collision nondegenerate condition, we have
\bea i(V_n;l_+)=i(V_d;l_+)+\phi(R) \lb{ndr}\eea
\end{prop}
\begin{proof}  Please note that $i(V_n;l_+)=i(V_d;l_+)+s(V_n,V_d,V^-_u,V^+_u)$,  we only need to show 
\bea s(V_n,V_d,V^-_u,V^+_u)=\phi(R).  \lb{3.15pf} \eea   We choose basis such that $R=diag(r_1,\cdots,r_n)$ and  consider the $2$-dimension  linear system  $\dot{z}=J_2\hat{B}_2(r)z$;
 (\ref{3.15pf}) is from
(\ref{s3}).

\end{proof}
Now we give the proof  of Theorem \ref{thm1.2}.
\begin{proof} Since $R$ is nondegenerate,
$V_u,V_d\pitchfork V^\pm_u$, from (\ref{pr1}), \bea
\lim_{e\to1}\mu(V_d,\hat{\ga}_e(\tau) V_d, \tau\in[0,\mathcal{T}]
)=i_-(V_d;l^+_0)+i_+(V_d,V_d;l^-_0)+i(V_d;l_+),\nonumber \eea
from (\ref{i01}-\ref{i02}), we have
\bea \lim_{e\to1}\mu(V_d,\hat{\ga}_e(\tau) V_d, \tau\in[0,\mathcal{T}])=k+i(V_d;l_+). \nonumber\eea 
Then,  (\ref{prd}) follows from Lemma \ref{lem2.3} and
the fact that $\ga_e(2\pi)=\hat{\ga}_e(\mathcal{T})$.

Similar, \bea
\lim_{e\to1}\mu(V_n,\hat{\ga}_e(\tau) V_n, \tau\in[0,\mathcal{T}] )=
i_-(V_n;l^+_0)+i_+(V_n,V_n;l^-_0)+i(V_n;l_+),\lb{prn1}\eea
(\ref{prn}) is from (\ref{i01}), (\ref{i03}-\ref{ndr}) and Lemma \ref{lem2.3}.
\end{proof}
It is worth noticing that in several cases  we can't analytically compute the 
collision index and for this reason, we now introduce a numerical method 
which can be useful in this situation. 

We first consider the case for
$\mathbb{R}^+$. We choose $V\in Lag(2n)$, such that $\hat{B}(\tau)|_V>0$ for
$\tau\in \mathbb{R}^\pm$. Then the crossing form $ \Gamma(\Lambda(\tau),V,\tau)>0$ and  
we have \bea
\mu(V,\ga(\tau)V_0)=\sum_{0<\tau_j<\infty} \nu(\tau_j),\nonumber\eea where
$\nu(\tau_j)=\dim V\cap \ga(\tau_j)V_0$. For the Lagrangian system,
we can always choose $V_d=V$ and then $\hat{B}(t)|_{V_d}=I_n>0$.
We can  get  the Maslov index from the H\"{o}rmander index, in fact
\bea\mu(V_1,\ga(\tau)V_0,
\tau\in[0,T])=\mu(V_d,\ga(\tau)V_0,[0,T])+s(V_1,V_d,V_0,\ga(T)V_0). \nonumber
\eea Under the nondegenerate conditions,
$\lim_{T\rightarrow\infty}\ga(T)V_0=V^+_u$, then we have
\bea\mu(V_1,\ga(\tau)V_0,
\tau\in\mathbb{R}^+)=\mu(V_d,\ga(\tau)V_0,\tau\in\mathbb{R}^+)+s(V_1,V_d,V_0,V^+_u).\nonumber
\eea The cases of  $\mathbb{R}^-$ and $\mathbb{R}$ are similar, so we
just make $-T$ be the starting time, where $T>0$ is large enough.

\begin{rem} \label{r3.e} For computing  the Maslov index $\mu(V_d,\ga(\tau)V_0)$, we 
start by choosing a basis $\{\xi_1(0),\cdots,\xi_n(0)\}$ of $V_0$ and,  
by using a numerical integrator, we get $\xi_k(t):=\ga(t)\xi_k(0)$. For
$j=1,\cdots,n$,  let $e_j$ the  basis of $V_d$, $t \mapsto M(t)$ be the path of
$2n\times2n$ matrices defined by 
$M(t):=(e_1,\cdots,e_n,\xi_1(t),\cdots,\xi_n(t))$ and we set $f(t)=\det(M(t))$. 
Then $\mu(V_d,\ga(t)V_0)$ is equal to the total number of zeros of
$f(t)$.  Since $t\mapsto|f(t)|$ is exponentially increasing, this method works well 
for not so large time. 
Instead, we use the robust numerical algorithm based on exterior algebra representation. 
We refer the interested reader to the following papers 
\cite{CDB1,CDB2, CDB3} in which the authors compute the Maslov index for homoclinic 
orbits. (Cfr. Section \S 6). \end{rem}

\section{Collision index for brake symmetry Central configurations}

In the next, we will consider the case that the central
configuration with brake symmetry. We start with the following Definition. 
\begin{defi}\lb{defs} The central configuration with normalized Hessian $R$ is said with
brake symmetry if there exists a $k\times k$ symmetry matrix $N$
which satisfies $N^2=I_k$,  $N{\mathbb  J}=-{\mathbb  J}N$ ,    $RN=NR$.
\end{defi}
To our knowledge, the  Lagrangian configuration  \cite{HLS} , Euler
central configuration \cite{LZhou}
 and the $1+n$ central configurations have the brake symmetry
 property \cite{M2},\cite{R3}. It could be interesting to find 
 an example without having this  symmetry property.

 In the brake symmetry case, let $\hat{N}=diag(N,-N)$, and denote
\bea g: x(\tau)\to \hat{N}x(\mathcal{T}-\tau). \eea Obviously,
$g^2=id$ and $g\cdot-J\frac{d}{d\tau}=-J\frac{d}{d\tau} \cdot g$. From the fact of
$q(\tau)=q(\mathcal{T}-\tau)$, $Q(\tau)=-Q(\mathcal{T}-\tau)$,  easy
computations show that
$\hat{N}\hat{B}(\mathcal{T}-\tau)=\hat{B}(\tau)\hat{N}$, and
consequently \bea g \hat{B}=\hat{B} g.  \eea Let $E^\pm=\ker(g\mp
I)$, then \bea \ker(-J\frac{d}{d\tau}-\hat{B}
)=\ker((-J\frac{d}{d\tau}-\hat{B})\big|_{E^+})\oplus \ker((-J\frac{d}{d\tau}-\hat{B})\big|_{E^-}).
\lb{kd} \eea Moreover, by the generalized Bott-type iteration
formula for Maslov index \cite{HS}( Th1.1) or \cite{LZZ}, we have
\bea i_1(\hat{\ga})+k=\mu(V^+(\hat{N}),\hat{\ga}(\tau)V^+(\hat{N}),
\tau\in[0,\mathcal{T}/2])+\mu(V^-(\hat{N}),\hat{\ga}(\tau)V^-(\hat{N}),
\tau\in[0,\mathcal{T}/2]), \lb{d1}\eea
 \bea i_{-1}(\hat{\ga})=\mu(V^+(\hat{N}),\hat{\ga}(\tau)V^-(\hat{N}), \tau\in[0,\mathcal{T}/2])+
 \mu(V^-(\hat{N}),\hat{\ga}(\tau)V^+(\hat{N}), \tau\in[0,\mathcal{T}/2]),\lb{d2} \eea
 where $V^\pm(\hat{N})=\ker(\hat{N}\mp I_{2n})$, $k=2n-4$.
Similarly,  we can   decompose of Dirichlet and Neumann boundary
condition as follows \bea
\mu(V_d,\hat{\ga}(\tau)V_d,\tau\in[0,\mathcal{T}])=\mu(V^+(\hat{N}),\hat{\ga}(\tau)V_d,
\tau\in[0,\mathcal{T}/2])+\mu(V^-(\hat{N}),\hat{\ga}(\tau)V_d),
\tau\in[0,\mathcal{T}/2]),\lb{d3} \eea \bea
\mu(V_n,\hat{\ga}(\tau)V_n,\tau\in[0,\mathcal{T}])=\mu(V^+(\hat{N}),\hat{\ga}(\tau)V_n),
\tau\in[0,\mathcal{T}/2])+\mu(V^-(\hat{N}),\hat{\ga}(\tau)V_n),
\tau\in[0,\mathcal{T}/2]). \lb{d4} \eea
 For the  iteration formula of brake orbits we refer the interested reader to \cite{LiuZ}.   Then we consider the
collision orbit, and let $\mathcal{K}$ be the space of bounded solution
of $\dot{z}=J\hat{B}(\tau)z$ on $l_+$, and $\mathcal{K}\pm$ be the
space of bounded solution on $l_+^-$ which satisfies $z(0)\in
V^\pm(\hat{N})$. Similar to (\ref{kd}), we have

\begin{lem} For the brake symmetry central configurations,  
on $l_+$, we have \bea \mathcal{K}=\mathcal{K}_+\oplus\mathcal{K}_-.\lb{k.1} \eea

\end{lem}
\begin{proof}
Please note that on $l_+$,
$\hat{N}\hat{B}(-\tau)=\hat{B}(\tau)\hat{N}$,  if $z(\tau)$ is one
solution then $\hat{N}z(-\tau)$ is another solution. Let
$z_\pm(\tau)=\frac{1}{2}(z(\tau)\pm\hat{N}z(-\tau))$, then
$z_\pm\in\mathcal{K}_\mp$, which implies the result.
\end{proof}

Obviously, there is standard brake symmetry on  $l_0$, that is
$\hat{N}_0=diag(I_k,-I_k) $, then $V^+(\hat{N}_0)=V_d$ and
$V^-(\hat{N}_0)=V_n$. Let $\mathcal{K}_0$ be the space of bounded
solution of $\dot{z}=J\hat{B}(\tau)z$ on $l_0$, and
$\mathcal{K}^0\pm$ be the space of bounded solution on $l_0^-$ which
satisfies $z(0)\in V^\pm(\hat{N}_0)$. We have  on $l_0$, \bea
\mathcal{K}_0=\mathcal{K}^0_+\oplus\mathcal{K}^0_-. \lb{dl0} \eea

On $l_0$, the nondegenerate condition is clear. Now, from (\ref{dl0}) and
Lemma \ref{lem4.2},  we have
\begin{prop}\lb{prop3.3} The system on $l_0$ is nondegenerate if and only if $R$
is nondegenerate.
\end{prop}

Since $R$ satisfies brake symmetry, let $V^\pm=\ker(N\mp
I_k)$, so it is obvious $dim V^\pm=\frac{k}{2}$. Let
 $\hat{V}^\pm=JV^\pm\oplus V^\pm$ be
symplectic subspace thus
$\mathbb{R}^{2k}=\hat{V}^+\oplus\hat{V}^-$. Set
$R^\pm=R|_{V^\pm}$,
and denote $ \hat{B}_\pm= \left( \begin{array}{cccc} I_\frac{k}{2} & \frac{Q}{4}I_\frac{k}{2} \\
\frac{Q}{4}I_\frac{k}{2} & -R^\pm \end{array}\right)$, we get 
$ \hat{B}=\hat{B}_+\diamond\hat{B}_-$, and the
fundamental solution satisfies
$\Psi_0=\Psi_0|_{\hat{V}^+}\diamond\Psi_0|_{\hat{V}^-} $. Furthermore 
we have \bea
i(V,W;l_0^-)=i(V|_{V_+},W|_{V_+};l_0^-)+i(V|_{V_-},W|_{V_-};l_0^-),
\lb{4.16} \eea for $V,W=V^\pm(\hat{N})$ or $V_d, V_n$.
 Please note that
\bea V^+(\hat{N})|_{\hat{V}^+}=V_d,\,\
V^+(\hat{N})|_{\hat{V}^-}=V_n,\,\ V^-(\hat{N})|_{\hat{V}^+}=V_n,\,\
V^-(\hat{N})|_{\hat{V}^-}=V_d. \lb{4.15} \eea
From (\ref{i02})-(\ref{i03}),
 we have
 \bea i_+(V^+(\hat{N}),V_d;l^-_0)+i_+(V^-(\hat{N}),V_d;l^-_0)=k, \lb{4.19} \eea 
 \bea i_+(V^+(\hat{N}),V_n;l^-_0)+i_+(V^-(\hat{N}),V_n;l^-_0)=2\phi(R), \lb{4.19.0}\eea
  \bea i_+(V^+(\hat{N}),V^+(\hat{N});l^-_0 )+i_+(V^-(\hat{N}),V^-(\hat{N});l^-_0 )=k+\phi(R).\lb{4.19.1} \eea
 \bea i_+(V^-(\hat{N}),V^+(\hat{N});l^-_0 )+i_+(V^+(\hat{N}),V^-(\hat{N});l^-_0 )=\phi(R).\lb{4.20} \eea
For the brake symmetry central configurations, we get the following
approximation theorem.
\begin{thm}\lb{th4.1} Let $\lambda_{1}(R)>-\frac{1}{8}$ be  
nondegenerate with brake symmetry property, 
and satisfying  the collision nondegenerate conditions.  
Thus we have \bea
 \lim_{e\to1}\mu(V_d,\hat{\ga}_e(\tau)V_d,\tau\in[0,\mathcal{T}])=
 k+i_-(V^-(\hat{N});l_+^-)+i_-(V^+(\hat{N});l_+^-),
 \lb{f1} \eea
 \bea
 \lim_{e\to1}\mu(V_n,\hat{\ga}_e(\tau)V_n,\tau\in[0,\mathcal{T}])=2\phi(R)+
 i_-(V^-(\hat{N});l_+^-)+i_-(V^+(\hat{N});l_+^-), \lb{f2}\eea
 \bea \lim_{e\to1}i_{-1}(\hat{\ga}_e)=\lim_{e\to1}i_1(\hat{\ga}_e)=\phi(R)+
i_-(V^-(\hat{N});l_+^-)+i_-(V^+(\hat{N});l_+^-).\lb{f3} \eea
\end{thm}
\begin{proof} Since $R$ is collision nondegenerate, from
(\ref{k.1}), the system is nondegenerate with respect to
$V^\pm(\hat{N})$ on $l^-_+$, and also from the condition that $R$ is
nondegenerate, then the system is nondegenerate with respect to
$V^\pm(\hat{N})$, $V_n$, $V_d$ on $l^-_0$. From (i) of Theorem
\ref{th.pr}, we have \bea
\lim_{e\to1}\mu(V,\hat{\ga}_e(\tau)W,\tau\in[0,\mathcal{T}/2])=i(V,W;l^-_0)+i(V;l^-_+), \nonumber
\eea for $V,W$ is $V^\pm(\hat{N})$, $V_n$, $V_d$. From
(\ref{d1}-\ref{d4}), we have \bea
\lim_{e\to1}\mu(V_d,\hat{\ga}_e(\tau)V_d,\tau\in[0,\mathcal{T}])=i_+(V^+(\hat{N}),V_d;l^-_0)
+i_+(V^-(\hat{N}),V_d;l^-_0)+i_-(V^+(\hat{N});l^-_+)+i_-(V^-(\hat{N});l^-_+),\nonumber\lb{p1}
\eea
 \bea \lim_{e\to1}\mu(V_n,\hat{\ga}_e(\tau)V_n,\tau\in[0,\mathcal{T}])=i_+(V^+(\hat{N}),V_n;
 l^-_0)+i_+(V^-(\hat{N}),V_n;l^-_0)+i_-(V^+(\hat{N});l^-_+)+i_-(V^-(\hat{N});l^-_+),\nonumber\lb{p2} \eea
 \bea \lim_{e\to1}i_1(\hat{\ga}_e)= i_+(V^+(\hat{N}),V^+(\hat{N});l^-_0 )+i_-(V^+(\hat{N});l^-_+)
 +i_+(V^-(\hat{N}),V^-(\hat{N});l^-_0 )+i_-(V^-(\hat{N});l^-_+)-k, \nonumber\lb{p3}\eea
 \bea \lim_{e\to1}i_{-1}(\hat{\ga}_e)= i_+(V^-(\hat{N}),V^+(\hat{N});l^-_0 )+i_-(V^+(\hat{N});l^-_+)
 +i_+(V^+(\hat{N}),V^-(\hat{N});l^-_0 )+i_-(V^-(\hat{N});l^-_+). \nonumber\lb{p4}\eea
Then (\ref{f1}-\ref{f3}) is from (\ref{d1}-\ref{d4}) and
(\ref{4.19}-\ref{4.20}).
\end{proof}
Comparing  of (\ref{prd}) in Theorem \ref{thm1.2} and (\ref{f1}), we
have
 \begin{cor}\label{cor4.7} 
 Let $\lambda_{1}(R)>-1/8$. If $R$  is nondegenerate, has 
 the brake symmetry property and if  the collision 
 nondegenerate condition is fulfilled, then  
 we have
 \bea i_-(V^-(\hat{N});l_+^-)+i_-(V^+(\hat{N});l_+^-)=i(V_d;l_+).\lb{lequ} \eea
\end{cor}
It is clear that $i(V_d;l_+)\geq0$.  We shall now prove 
that $i_-(V^\pm(\hat{N});l_+^-)$ is also nonnegative.
  We consider the Maslov index  on $\R^-$,
supposing the system is
 nondegenerate with respect to $V_1$, and $V_1\pitchfork V_u^-$,      then for $-\tau_0$ large enough,
  \bea \mu(V_1,V_u(\tau),
\tau\in(-\infty,0])=\mu(V_1,\ga(\tau,\tau_0)V_u(\tau_0),
\tau\in[\tau_0,0]).\nonumber\eea From the property (III), (IV), for $\ga(-s,0)$,
$s\in [0,\infty)$

\bea \mu(V_1,\ga(\tau,\tau_0)V_u(\tau_0),
\tau\in(\tau_0,0))=\mu(\ga(-s,0)V_1,V^-_u, s\in [0,\infty)), \nonumber\eea we
have \bea \mu(V_1,V_u(\tau), \tau\in(-\infty,0])=-\mu(V^-_u,\ga(-\tau)
V_1, \tau\in[0,\infty) ).\nonumber \eea By the nondegenerate condition, we
have $\lim_{T\rightarrow\infty}\ga(-T)V_1=V_s$, then we have \bea
 \mu(V_1,V_u(\tau), \tau\in(-\infty,0])= s(V_d,V^-_u,V_1,V^-_s)-\mu(V_d,\ga(-\tau)V_1, 
 \tau\in[0,+\infty)).\nonumber\eea

In the case of ERE, an easy
computation shows that \bea
\frac{d}{d\tau}\Psi_+(-\tau)=-J\hat{B}(-\tau)\Psi_+(-\tau), \lb{l1.0} \eea
where $\Psi_+(\tau)=\Psi_+(\tau,0)$ is the fundamental solution on $l_+$.
If the central configuration  satisfies the brake symmetry, that is
$\hat{N}\hat{B}(-\tau)=\hat{B}(\tau)\hat{N}$, then, direct computation
shows that \bea\Psi_+(-\tau)=\hat{N}\Psi_+(\tau)\hat{N},
\tau\in[0,\infty). \nonumber\eea So we have \bea
\mu(V_d,\Psi_+(-\tau)V_1)=\mu(V_d,\hat{N}\Psi_+(\tau)\hat{N} V_1,
\tau\in[0,\infty) ). \nonumber\eea Please note that if
$-\hat{B}(\tau)|_V<0$ for $t\in \mathbb{R}^+$, then the crossing form $
\Gamma(\Lambda(t),V_d,t)<0$, we have \bea
\mu(V_d,\Psi_+(-\tau)V_1)=-\sum_{0<\tau_j<\infty} \nu(\tau_j)\leq0,\nonumber\eea
where $\nu(\tau_j)=\dim V_d\cap \Psi_+(-\tau_j)V_1$.

Please note that,  in the case $V_1=V^j_d\oplus V^{(k-j)}_n$,  
where $V^j_d\in V_d, V^{(k-j)}_n\in V_n$ ,
from
(\ref{s4}),  we have \bea s(V_d,V^-_u,V_1,V^-_s)=0.\nonumber \eea
Since $V^\pm(\hat{N})$ is a direct sum of Dirichlet Lagrangian subspace
and Neumann Lagrangian subspace,  by (\ref{4.15}),  we have
 \begin{lem}\label{lem4.9} On $l^-_+$, we have 
 \bea i_-(V^\pm(\hat{N}))=\sum_{0<\tau\pm_j<\infty}\nu(\tau^\pm_j)\geq0, \lb{pos}
 \eea where $\nu(\tau^\pm_j)=\dim V_d\cap \Psi_+(-\tau^\pm_j)V^\pm(\hat{N})$.
\end{lem}

 \section{Applications}
We give applications for the collision index. In subsection \S 5.1
we consider the ERE of minimal central configurations and prove  some
hyperbolicity results.  At \S 5.2, we study the stability of Euler
orbits.

 \subsection{Minimal central configurations}

In order to give a hyperbolic criteria, we first review some
results on Morse index.  Consider the linear Sturm systems
\begin{eqnarray}
-\frac{d}{dt}(P(t)\dot{y}+Q(t)y)+Q^T(t)\dot{y}+R(t)y=0,\lb{n2.7}
\end{eqnarray}
as $P,R,Q$ are continuous path of matrices in $\mathbb{R}^{2n}$
and satisfy  $P(t)>0$, $R(t)=R(t)^T$.  This linear Sturm system
(\ref{n2.7}) corresponds to the linear Hamiltonian system \bea
\dot{z}=JB(t)z, z\in {\R}^{2n},\lb{n2.8}\eea where \bea B(t)=\left(
\begin{array}{cccc} P^{-1}(t) & -P^{-1}(t) Q(t)\cr -Q^T P^{-1}(t) &
Q^T(t)P^{-1}(t)Q(t)-R(t)\end{array}\right). \lb{n2.81} \eea Let \bea
L(t,x(t),\dot{x}(t))={\frac12}((P\dot{x}+Qx)\cdot\dot{x}+Q^T\dot{x}\cdot
x+Rx\cdot x),\lb{nn3.1} \eea
 and ${\cal F}(x)=\int_0^{T}\{L(t,x(t),\dot{x}(t))\}dt$ on 
 $W^{1,2}([0,T],\mathbb  C^n)$.  We denote  \bea D(\omega,T)=\{x\in W^{1,2}([0,T],\mathbb{C}^n),
 x(0)=\omega x(T) \}, \omega\in {\mathbb  U}. \nonumber \eea
Obviously, \bea W^{1,2}_0([0,T],\R^n)\subset D(\omega,T)\subset
W^{1,2}([0,T],\mathbb  C^n).\nonumber \eea

Let $\mathcal{L}={\cal
F}''(0):=-\frac{d}{dt}(P(t)\frac{d}{dt}+Q(t))+Q^T(t)\frac{d}{dt}+R(t)$,
and  more precisely, set $\mathcal{L}_n, \mathcal{L}_\omega,
\mathcal{L}_d$ to be the operator with form $\mathcal{L}$ under the
Neumann, $\omega$ and Dirichlet boundary conditions separately. Let
$\lambda_k(\mathcal{L})$ be the $k$-th eigenvalue of $\mathcal{L}$.
 From the monotonicity property of the eigenvalues \cite{CH},
we have \bea \lambda_k(\mathcal{L}_n)\leq
\lambda_k(\mathcal{L}_\omega)\leq \lambda_k(\mathcal{L}_d).
\lb{eig}\eea Let $\phi$ be the Morse index of $\mathcal{L}$ which is
defined to be the total number of negative eigenvalues, which is
equal  to  the dimension of maximum negative definite subspace of ${\cal
F}$. Let $\phi_d, \phi_\omega, \phi_n$ be the Morse index of
$\mathcal{L}_d,\mathcal{L}_\omega,\mathcal{L}_n$ separately.  From
(\ref{eig}), we have \bea \phi_d\leq \phi_\omega\leq \phi_n. \nonumber\eea

\begin{prop}\label{prop4.1}  The system is hyperbolic  if $\phi_n=\phi_d$ and $\mathcal{L}_n$ is 
nondegenerate.  \end{prop}
\begin{proof} Please note that the system is hyperbolic,  that is,  $\sigma(\ga(T))\cap
\mathbb{U}=\emptyset$ is equivalent to $\mathcal{L}_\omega$ which is
nondegenerate for  $\forall\omega\in\mathbb{U}$. Supposing
$k_0=\phi_n=\phi_d$,  we have
$\lambda_{k_0}(\mathcal{L}_\omega)<0$ by  (\ref{eig}). On the other hand,
$\mathcal{L}_n$ is nondegenerate which implies
$\lambda_{k_0+1}(\mathcal{L}_n)>0$ and hence
$\lambda_{k_0+1}(\mathcal{L}_\omega)>0$, which implies the result.
\end{proof}
Please note that $\phi_n=0$ implies $\phi_d=0$, so we have
\begin{cor}\label{cor5.0}  The system is hyperbolic  if $\mathcal{L}_n>0$.
 \end{cor}

From Theorem 1.2 of \cite{HS} or P172 of \cite{Lon4}, we list  the
relation of Morse index and Maslov index below.
\begin{lem}\label{lem5.5} Let $\ga$ be the fundamental solution of (\ref{n2.8}), then we have
 \bea
\phi_{\omega}(\mathcal{L})=i_{\omega}(\gamma), \ \
\nu_{\omega}(\mathcal{L})=\nu_{\omega}(\gamma),\,\ \forall\omega\in\mathbb U, \lb{inr1} \eea \bea
\phi_{d}(\mathcal{L})+n=\mu(V_d,\ga V_d),\ \
\phi_{n}(\mathcal{L})=\mu(V_n,\ga V_n).\lb{inr2} \eea

\end{lem}

{\bf Proof of Theorem \ref{thm1.3}.} Please note that the central configuration $a_0$ is 
non degenerate minimizer which implies
$\lambda_{1}(R)>0$, i.e. $\phi(R)=0$ and $R$ is nonsingular.  
Under the collision nondegenerate condition,  from (\ref{prd})-(\ref{prn}),  for
 $1-e$ small enough
 \bea \mu(V_d,\ga_e(t)V_d,t\in[0,2\pi])-k=\lim_{e\to1}\mu(V_n,\ga_e(t)V_n,  t\in[0,2\pi]).\lb{lequ.0} \eea
From (\ref{inr2}), we have $\phi_d=\phi_n$. The nondegenerate of
$\mathcal{L}_n $ is from Theorem \ref{thm1.2},  so  the result is from Proposition \ref{prop4.1}. $\square$

A typical example is the Lagrangian equilateral triangle central
configuration. It is obvious
that $R=diag((3+\sqrt{9-\bb})/2,(3-\sqrt{9-\bb})/2 ) $  satisfies the
brake symmetry with $N=diag(1,-1)$. This fact had been used to
decompose  the $-1$-degenerate curves in \cite{HLS}. It is proved
in \cite{HLS} that for any $\beta\in(0,9]$, $1-e$ small enough,
$\mathcal{L}_n$ is positive, and consequently  hyperbolic. By
the approximation formula (\ref{f2}) and the nonnegative property
(\ref{pos}), we have
\begin{prop}\label{prop4.3}  If the   Lagrangian  central configurations is collision nondegenerate, 
then $i_-(V^-(\hat{N});l^-_+)=i_-(V^+(\hat{N});l^-_+)=0$,
and hence $i(V_d;l_+)=0$. \end{prop}

We continue by  studying the case of strong minimizer, so please note that
 a central configuration is strong minimizer if it satisfies
$\lambda_{1}(R)>1$. The next lemma is important in the proof of Theorem
\ref{thm1.4}.

\begin{lem}\label{lem5.6}\textbf{(See \cite{HO},Proposition 2.)}
If $\delta>1$,$\omega\in \U$, then $\mathcal{A}(e,\delta)=-\frac{d^2}{dt^2}-1+\frac{\delta}{1+e\cos(t)}$ 
is positive operator for all
$e\in[0,1)$ on its domain $\bar{D}_{1}(\omega,2\pi)$, where $\bar{D}_{n}(\omega,2\pi)=\{y\in W^{2,2}([0,2\pi], 
\mathbb  C^{n})| y(2\pi)=\omega y(0),\dot{y}(2\pi)=\omega \dot{y}(0)\}$.
\end{lem}
Now we can proof Theorem \ref{thm1.4}.
\begin{proof}
For the ERE, we have
 \bea \mathcal{L}=-\frac{d^{2}}{dt^2}I_{k}-2\mathbb{J}_{k/2}\frac{d}{dt}+\frac{R}{1+e\cos(t)}. \nonumber\eea
then $\mathcal{L}>\hat{\mathcal{L}}:=-\frac{d^{2}}{dt^2}I_{k}-2\mathbb{J}_{k/2}\frac{d}{dt}+
\frac{\lambda_1(R)I_k}{1+e\cos(t)} $. 
We only need to show $\hat{\mathcal{L}}>0$ with domain $\bar{D}_{k}(\omega,2\pi)$
for any $\omega\in \U$.
   Let $\mathcal{R}(t)=\left(
                                                                                                                              \begin{array}{cc}
                                                                                                                                \cos(t)I_k & -\sin(t)I_k \\
                                                                                                                                \sin(t)I_k & \cos(t)I_k \\
                                                                                                                              \end{array}
                                                                                                                            \right)
 $, then
 \bea \mathcal{R}\hat{\mathcal{L}}\mathcal{R}^T=-\frac{d^{2}}{dt^2}I_{k}-I_k+
 \frac{\lambda_1(R)I_k}{1+e\cos(t)}. \nonumber    \eea
Since $\lambda_{1}(R)>1$,we get the result  from Lemma \ref{lem5.6}.
\end{proof}
It is obvious that for the strong minimizer if it is collision
nondegenerate, the approximation theorem implies $ i(V_d,l_+)=0$.

In the special case, the ERE of Lagrangian central configurations is
hyperbolic for $\beta>8$, $e\in[0,1)$,  which had proved directly in
\cite{Ou}. As another  example, we consider the $1+3$ central
configurations, and let $m_1=m_2=m_3=1$ and $m_0=m_c$, the essential
part $R=I_4+\mathcal{D}$ with   \bea \mathcal
{D}=\left(\begin{array}{ccccc}\frac{1}{2}
& 0 & -\frac{3\sqrt{3u(3+m_c)}}{2(1+\sqrt{3}m_c)} & 0\\
                       0 & \frac{1}{2} & 0 &  \frac{3\sqrt{3m_c(3+m_c)}}{2(1+\sqrt{3}m_c)} \\
                      -\frac{3\sqrt{3m_c(3+m_c)}}{2(1+\sqrt{3}m_c)} & 0&\frac{\sqrt{3}
                      (3+m_c)}{2(1+\sqrt{3}m_c)} & 0 \\
                       0 &  \frac{3\sqrt{3m_c(3+m_c)}}{2(1+\sqrt{3}m_c)} &0 &
                       \frac{\sqrt{3}(3+m_c)}{2(1+\sqrt{3}m_c)}
                       \end{array}\right),       \eea
                      was  computed in \cite{MS}. Please note that
                       there is a typo in (39) of
                       \cite{MS}.
Let $$\mathcal {D}_\mp=\left(\begin{array}{cc}\frac{1}{2} & \mp\frac{3\sqrt{3m_c(3+m_c)}}{2(1+\sqrt{3}m_c)}  \\
                                    \mp\frac{3\sqrt{3m_c(3+m_c)}}{2(1+\sqrt{3}m_c)} & \frac{\sqrt{3}(3+m_c)}
                                    {2(1+\sqrt{3}m_c)}
                                    \end{array}\right),$$
 then        $\mathcal {D}=\mathcal {D}_-\diamond\mathcal {D}_+$.
  Obviously the eigenvalues $\lambda_\pm$ of $\mathcal {D}_\mp$ is same,  direct
  computation shows that
\bea
\lambda_\pm(m_c)=\frac{1}{2}(1+\sqrt{3}m_c)^{-1}\[\sqrt{3}m_c+\frac{3\sqrt{3}+1}{2}
\pm\(27(m_c^2+3m_c)+\(\frac{3\sqrt{3}-1}{2}\)^2\)^{\frac{1}{2}}\].
\eea Obviously $\lambda_+(m_c)>0$ for $m_c\in[0,+\infty)$. Let
$m_c(0)=\frac{\sqrt{3}}{24}$, $m_c(-1)=\frac{81+64\sqrt{3}}{249}$,
then $\lambda_-(m_c(0))=0$ and $\lambda_-(m_c(-1))=-1$, moreover
\bea \left\{\begin{array}{ll} \lambda_-(m_c)>0 & \quad
           {\mathrm if}\; m_c\in [0,m_c(0)),  \\
           \\
       -1<\lambda_-(m_c)<0  & \quad
           {\mathrm if}\; m_c\in (m_c(0),m_c(-1)),              \\
           \\
 -\frac{9}{8}<\lambda_-(m_c)<-1, & \quad {\mathrm if}\; m_c\in (m_c(-1),+\infty).\end{array}\right. \lb{5.16} \eea
Since $\lambda_1(R)=1+\lambda_-(m_c)$, it is obvious that Theorem
\ref{thm1.3} and Theorem \ref{thm1.3} imply Corollary
\ref{cor1.5}.

\begin{rem} \label{r5.1}
Inspired by the recent results obtained in \cite{HLS},  we conjecture
that the nondegenerate  minimal central configuration is  collision
nondegenerate  and satisfies \bea i(V_d;l_+)=0.\eea

In a private communication with the first name author, prof. Y. Long posed the 
following conjecture.
\begin{quote}
A smooth $T$-periodic non-collision solution of the planar $N$-body problem, with $N > 3$, is a  smooth global
 minimizer of the action functional on the space of all $T$-periodic orbits having non-trivial winding number
  if and only if it is an elliptic motion corresponding to the global minimal central
 configuration of the potential restricted to the inertia ellipsoid. 
\end{quote}
We observe that, Long's conjecture implies  that,  for any $e\in[0,1)$  
the Morse index for  the ERE of  any non-degenerate minimal central configuration is
$0$.  In the case of brake symmetric  central
configurations, Long's conjecture, imply our conjecture in
the collision nondegenerate case.
This conjecture is still open, and  in the case $e=0$, an interesting 
result from  Chenciner and  Desolneux   shows that  the  minima of the 
action on the zero mean loop space,  is the relative equilibrium  
corresponding to a minimal central configuration.
\cite{CD}. 
\end{rem}

\subsection{Stability analysis of Euler orbits}

The Euler orbits have been studied in \cite{MSS1}, \cite{LZhou}, in this
case,  $R=diag(-\delta,2\delta+3)$, where $\delta\in[0,7]$ only
depends on mass $m_1,m_2,m_3$. Please refer to Appendix A of \cite{MSS1}
for the details.  Although there is no physical meaning for $\delta>7$, we will
assume $\delta\geq 0$ to make the mathematical theory complete.

We will use the index theory to study the stability problem. Let $\ga_{\delta,e}$ be the fundamental solutions of 
$\mathcal{B}(t)$ which is given by (\ref{msf}), that is 
$\dot{\ga}_{\delta,e}=J_2\mathcal{B}(t)\ga_{\delta,e}$, $t\in[0,2\pi]$, $\ga_{\delta,e}(0)=I_4$.   
For $\delta=0$, the system  degenerates to the Kepler problem, and it
has been studied in \cite{HS1}, which  proved 
that $\ga_{0,e}(2\pi)$ with normal form $\left(\begin{array}{cc} 1 & 1 \\ 0 & 1 \end{array}\right)\diamond
I_2$ and
\bea i_{\omega}(\ga_{0,e})=\left\{\begin{array}{ll} 0, & \quad
           {\mathrm if}\; \omega=1,  \\
           \\
 2, & \quad {\mathrm if}\; \omega\in\U\setminus \{1\} .\end{array}\right.\lb{e1}\eea
The Maslov-type index in the case $\delta>0$ has been studied by Long and Zhou \cite{LZhou},  then we
review firstly their results. For any $j\in\mathbb{N}$, there exists
$1$-degenerate curves $\Gamma_j=Gr(\vf_j(e))$, and we also let
$\Gamma_0=Gr(\vf_0(e))$ with $\vf_0(e)=0$.   Then $\ga_{\delta,e}$
only degenerates at $ \cup_{j=1}^\infty \Gamma_j $ and
$dim\ker(\ga_{\delta,e}(2\pi)-I_4)=2 $ if
$(\delta,e)\in\cup_{j=1}^\infty \Gamma_j$. The Maslov-type index
satisfies \bea i_1(\ga_{\delta,e})=2j+3, \quad if\quad
\vf_j(e)<\delta\leq\vf_{j+1}(e),\quad j\in \mathbb{N}\cup\{0\}.
\eea  Similarly, for $\forall j\in\mathbb{N}$, there exists pair
$-1$-degenerate curves $\Upsilon_j^\pm=Gr(\psi_j^\pm(e))$.

 Let $\psi_j^s(e)=min\{\psi_j^+(e),\psi_j^-(e)\}$ and
$\psi_j^l(e)=max\{\psi_j^+(e),\psi_j^-(e)\}$. Moreover, we set
$\psi_0^l=\psi_0^s=0$,  then for $k\in\mathbb{N}$ we have \bea
i_{-1}(\ga_{\delta,e})=\left\{\begin{array}{ll} 2j, & \quad
           {\mathrm if}\; \delta\in (\psi_{j-1}^l,\psi_{j}^s],  \\
           \\
 2j+1, & \quad {\mathrm if}\; \delta\in(\psi_j^s,\psi_{j}^l] .\end{array}\right.\lb{l3}\eea
Direct computation shows that $\psi_j^+(0)=\psi_j^-(0)$, but it is not
clear if, for $e>0$,  there exist other intersection points.  There
is a monotonicity property for Maslov-type index, that is for
$\omega\in\mathbb{ U}$ \bea i_\omega(\ga_{\delta_1,e})\leq
i_\omega(\ga_{\delta_2,e}),\,\ if \,\ \delta_1\leq \delta_2.
\lb{mono}\eea The $\pm1$ degenerate curves satisfies \bea
0<\psi_1^s\leq\psi_1^l<\vf_1<\psi_2^s\leq\psi_2^l<\cdots
\psi_j^s\leq\psi_j^l<\vf_j<\psi_j^s\leq\psi_j^l<\cdots . \lb{e4}
\eea
Moreover for the region between the $\pm1$-degenerate curves, $\ga_{\delta,e}(2\pi)$ is elliptic-hyperbolic
and for  the region between the pairs of  $-1$-degenerate curves $\ga_{\delta,e}(2\pi)$
is hyperbolic.

As a continuous work of Long and Zhou \cite{LZhou}, we use the   near-collision index
 to study the limit case.
 From Theorem 1.1, we have
\begin{thm}\label{thm5.1} If $\delta\in(1/8,7]$,   let $\varepsilon=\frac{1}{2}\min \{\frac{\delta}{2\delta+5}, 1/8  \}$,  
$\hat{e}=\frac{1-e^2}{2}$. For $\hat{e}<\varepsilon^3$, we have   
\bea i_1(\ga_e)\geq \frac{2}{\pi}(\delta-\frac{1}{8})^{\frac{1}{2}} \ln\(\frac{\varepsilon^2}{\sqrt{\hat{e}}}\)-6. \lb{ifty}  \eea
Hence
\bea \sup\{
\overline{\lim}_{e\to1}\vf_j(e),\overline{\lim}_{e\to1}\psi^\pm_j(e),j\in\mathbb{N}\}\leq1/8.
\lb{th5.2.1} \eea
\end{thm}
\begin{proof} Since $R=diag(-\delta,2\delta+3)$, then for
$\delta>1/8$, $\lambda_1(R)<-1/8$. Now equation (\ref{ifty}) follows from (\ref{t1.01}) of Theorem
\ref{thm1.1}. To prove (\ref{th5.2.1}), let
$\hat{\delta}_j=\overline{\lim}_{e\to1}\vf_j(e)$, then there exists
$e_l\to1$ such that $\lim_{l\to\infty}\vf_j(e_l)=\hat{\delta}_j$. If
$\hat{\delta}_j>1/8$, then we choose
$\delta_\varepsilon\in(1/8,\hat{\delta}_j)$, for $l$ large enough,
$\delta_\varepsilon<\vf_j(e_l)$, by the monotone property
(\ref{mono}),  so we have $i_1(\ga_{\delta_\varepsilon,e_l})\leq 2j+1 $,
which contradict to (\ref{ifty}). The proof for
$\overline{\lim}_{e\to1}\psi^\pm_j(e)\leq1/8$  is similar.

\end{proof}

Let $N=diag(1,-1)$, then $NR=RN$, we will compute the collision
index for $\delta\in(0,1/8)$ by the decomposition property.  By the
brake symmetry, from (\ref{kd}) we have \bea
dim\ker(\ga_{\delta,e}(2\pi)+1)=dim(V^-(\hat{N})\cap\ga_{\delta,e}(2\pi)V^+
(\hat{N}))+dim(V^+(\hat{N})\cap\ga_{\delta,e}(2\pi)V^-(\hat{N})).\nonumber\eea
We always set $\psi_k^+$  to be the degenerate curve in the sense that
$V^+(\hat{N})\cap\ga_{\delta,e}(2\pi)V^-(\hat{N})$ nontrivial and
similarly $\psi_k^-$ to be the degenerate curve in the sense that
$V^-(\hat{N})\cap\ga_{\delta,e}(2\pi)V^+(\hat{N})$ nontrivial.

We get the collision index on $l^-_+$ numerically from (\ref{pos})
and the step in Remark \ref{r3.e}.  With the help of matlab, we have
\\
{\bf Numerical result A}: For the Euler orbits is collision
nondegenerate for  $\delta\in(0,\frac{1}{8})$, and on $l^-_+$  \bea
i_-(V^+(\hat{N});l^-_+)=i_-(V^-(\hat{N});l^-_+)=1. \lb{e10} \eea

It is obvious that $\phi(R)=1$ for $\delta\in(0,\frac{1}{8})$,
then from (\ref{f3}), we have

\begin{cor}\label{prop4.3.1} For $\delta\in(0,\frac{1}{8})$, under the condition of numerical result A, 
we have  \bea  \lim_{e\rightarrow1}i_1(\ga)=\lim_{e\rightarrow1}i_{-1}(\ga)=3. \lb{eee} \eea  \end{cor}
From (\ref{i02}-\ref{i03}), easy computation shows that for
$\delta\in(0,\frac{1}{8})$,  $$i(V^+(\hat{N}),V^-(\hat{N});l^-_0)=1,  \quad
i(V^-(\hat{N}),V^+(\hat{N});l^+_0)=0.$$ So for $1-e$ small
enough, \bea \mu(V^+(\hat{N}),\ga_{\delta,e}
V^-(\hat{N}))=2,\lb{ed1} \eea \bea \mu(V^-(\hat{N}),\ga_{\delta,e}
V^+(\hat{N}))=1. \lb{ed2}\eea

\begin{thm}\label{thm521} Under the assumption of  numerical fact A, for the $\pm1$-degenerate curve, we have
\bea
\lim_{e\rightarrow1}\vf_j(e)=\lim_{e\rightarrow1}\psi^\pm_{j+1}(e)=1/8,
\,\ for \,\ j\in\mathbb{N}, \lb{l1} \eea \bea
\lim_{e\rightarrow1}\psi^+_{1}(e)=0, \quad
\lim_{e\rightarrow1}\psi^-_{1}(e)=1/8. \lb{l2}\eea       \end{thm}
\begin{proof} To prove (\ref{l1}), from Theorem \ref{thm5.1}, we only need to
show \bea \inf\{
\varliminf_{e\rightarrow1}\vf_j(e),\varliminf_{e\rightarrow1}\psi^\pm_{j+1}(e), j\in\mathbb{N}\}\geq1/8.
\lb{th5.2.1.1} \eea The proof of (\ref{th5.2.1.1}) is similar to
(\ref{th5.2.1}).
 Let   $\bar{\delta}_j=\varliminf\limits_{e\to1}\vf_j(e)$. If $\bar{\delta}_j<1/8$,
 then we choose $e_l\to1$, such that $\vf_j(e_l)\to\bar{\delta}_j$.
 Choose $\varepsilon<1/8-\bar{\delta}_j$, for $l$ large enough,
 $\varphi_j(e_l)<\bar{\delta}_j+\varepsilon$, so we have
 \bea i_1(\ga_{\bar{\delta}_j+\varepsilon,e_l})> i_1(\ga_{\vf_j(e_l),e_l})=2j+3,\eea
 which is contradict  to (\ref{eee}). It is totally similar that
 $\varliminf\limits_{e\rightarrow1}\psi^\pm_{j+1}(e)\geq
 1/8$ for $j\in \mathbb{N}$.

Direct computation shows that for $\delta\in(0,1/8)$,
$\mu(V^-(\hat{N}),\ga_{\delta,0}
V^+(\hat{N}))=\mu(V^+(\hat{N}),\ga_{\delta,0} V^-(\hat{N}))=1$. By
monotone property,  \bea \mu(V^+(\hat{N}),\ga_{\delta,e}
V^-(\hat{N}))=\left\{\begin{array}{ll} 1, & \quad
           {\mathrm if}\; \delta\in (0,\psi^+_{1}],  \\
           \\
 2, & \quad {\mathrm if}\; \delta\in(\psi^+_{1},\psi_{2}^+] .\end{array}\right.\lb{l3.1}\eea
 From (\ref{ed1}), we get $\lim_{e\rightarrow1}\psi^+_{1}(e)=0$. The
 proof for $\lim_{e\rightarrow1}\psi^-_{1}(e)=1/8$ is from (\ref{ed2}) and the step is similar.
\end{proof}

This theorem shows that the system is hyperbolic for
$\delta\in(0,\frac{1}{8})$, and $1-e$ small enough.
To explain  the results, we use the following pictures which are taken from \cite{MSS2}.
\begin{figure}[H]
\begin{minipage}[t]{0.5\linewidth}
\centering
\includegraphics[height=0.73\textwidth,width=0.98\textwidth]{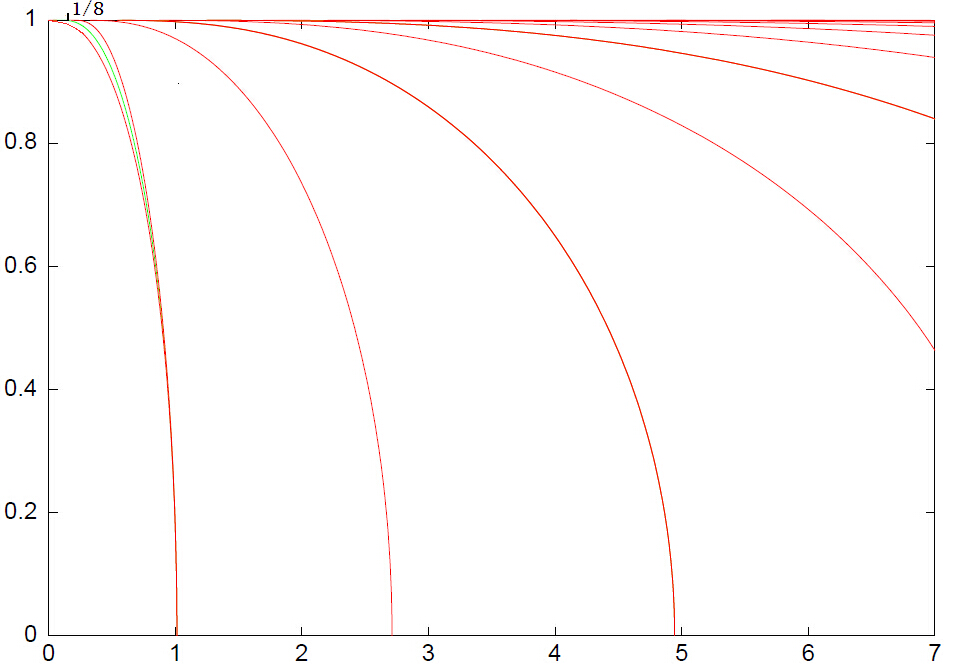}
     \caption{Stability bifurcation diagram. }
\label{fig:side:a4}
\end{minipage}%
\begin{minipage}[t]{0.5\linewidth}
\centering
\includegraphics[height=0.73\textwidth,width=1\textwidth]{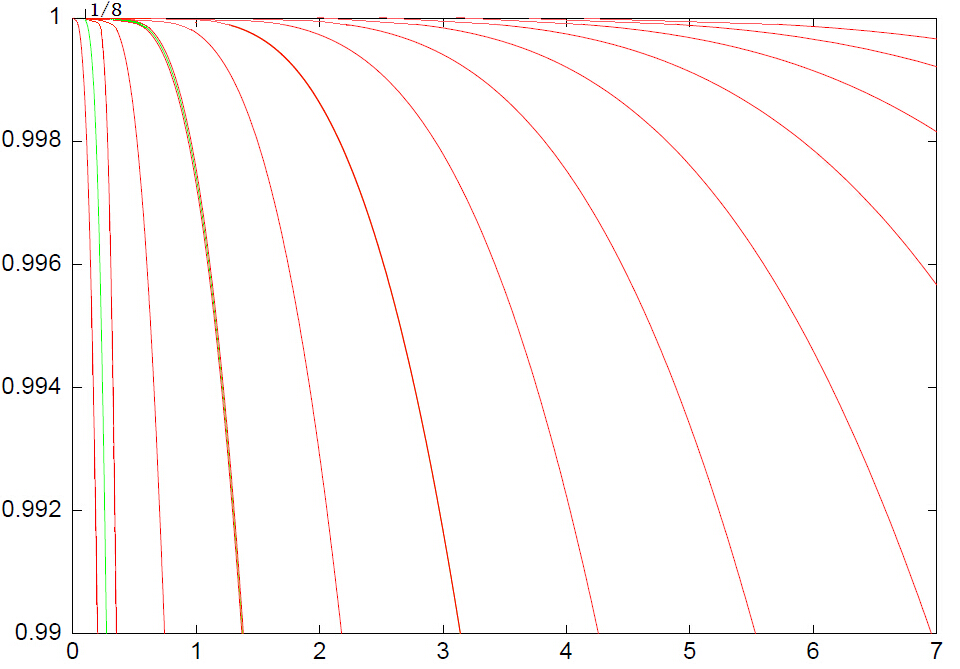}
     \caption{ A magnification of Figure 7 for $1-e$ small.}
\label{fig:side:b4}
\end{minipage}
\end{figure}

\section{Numerical results for collision index}

As shown in Remark  \ref{r3.e},  in order compute the collision index, 
we only need to count the zeros of a 
determinant function.  We use the exterior algebra representation from  \cite{CDB1,CDB2, CDB3}  
to do the computation.  
For reader's convenience, we give a brief review in the four-dimensional case, here.

Consider the linear system
\bea \dot{x}=A(\tau)x,\,\  x\in\mathbb R^4, \,\ \tau\in[0,+\infty),  \lb{6.1} \eea
where $A(+\infty)$ is hyperbolic.  Let $\wedge^2(\R^4)$ be the vector space of $2$-vector space in $\R^4$. 
Supposing $e_j, j=1,\dots,4$ is basis of $\R^4$, then $\hat{e}_1=e_1\wedge e_2$, $\hat{e}_2=e_1\wedge e_3$, 
$\hat{e}_3=e_1\wedge e_4$, $\hat{e}_4=e_2\wedge e_3$, $\hat{e}_5=e_2\wedge e_4$, $\hat{e}_6=e_3\wedge e_4$ 
is basis of $\wedge^2(\R^4)$.  There is a induced system from (\ref{6.1})
\bea   \dot{y}=A^{(2)}(\tau)y, \,\ y\in \wedge^2(\R^4). \lb{6.2}   \eea
Suppose $A=(a_{i,j})$, then $A^{(2)}$ could be expressed by $(a_{i,j})$. (Cfr. \cite[Equation (2.8)]{CDB3} for
the expression).  Let  $\sigma$ be the sum of the eigenvalues of $A(\infty)$ with positive real part.
Let $\hat{y}(\tau)=e^{-\sigma\tau}y(\tau)$,  then
\bea    {\frac{d\hat{y}}{d\tau}}=(A^{(2)}(\tau)-\sigma I_4 ) \hat{y}.         \lb{6.3}          \eea

 To compute the Maslov index $\mu(V_d,\ga(\tau)V_0)$,
we choose a basis $\xi_1(0),\xi_2(0)$ of $V_0$,  and 
let $\hat{y}(0)=y(0)=\xi_1(0)\wedge \xi_2(0)=\sum_{j=1}^6y_j(0)\hat{e}_j$. 
Then $\hat{y}(\tau)$ could be computed by matlab from Equation (\ref{6.3}). 
Let $\ga$ be the fundamental solution of (\ref{6.1}), 
then $\ga(\tau)V_0$ could be expressed by $\hat{y}(\tau)$.
We choose
$e_1,e_2$ to be the basis of $V_d$,  then  it is obvious that 
$V_d\cap\ga(\tau)V_0$ is nontrivial if and only if
$e_1\wedge e_2\wedge\hat{y}(\tau)=0$, which is equivalent to
 $\hat{y}_6(\tau)=0$.  So we can draw the picture of  $\hat{y}_6(\tau)$ 
 and count the number of zero points to get the Maslov index.

We will compute $ i_-(V^\pm(\hat{N});l^-_+)$ for Euler and Lagrangian orbits. From Lemma \ref{lem4.9},
 we only need to count the  points of $V_d\cap \ga(-\tau^\pm_j)V^\pm(\hat{N})$.
From (\ref{l1.0}),  the linear system with form  with $\dot{x}(\tau)=-J\hat{B}(-\tau)x(\tau) $,  
let $A(\tau)=-J\hat{B}(-\tau)$, 
then we can get $A^{(2)}(\tau)$.  We choose $e_1,e_4$ to be the basis of $V^+(\hat{N})$
and  $e_2, e_3$ to be the basis of $V^-(\hat{N})$. Let $\hat{y}^+(\tau)$ 
be the solution of (\ref{6.3}) with initial 
condition $\hat{y}(0)=\hat{e}_3$ and  $\hat{y}^-(\tau)$ be the solution 
with initial condition $\hat{y}(0)=\hat{e}_4$, 
then $ i_-(V^\pm(\hat{N});l^-_+)$ just is the zero points of $\hat{y}^\pm_6(\tau)$.

We firstly give some numerical pictures for Euler orbits:
\begin{figure}[H]
\begin{minipage}[t]{0.5\linewidth}
\centering
\includegraphics[height=0.73\textwidth,width=0.98\textwidth]{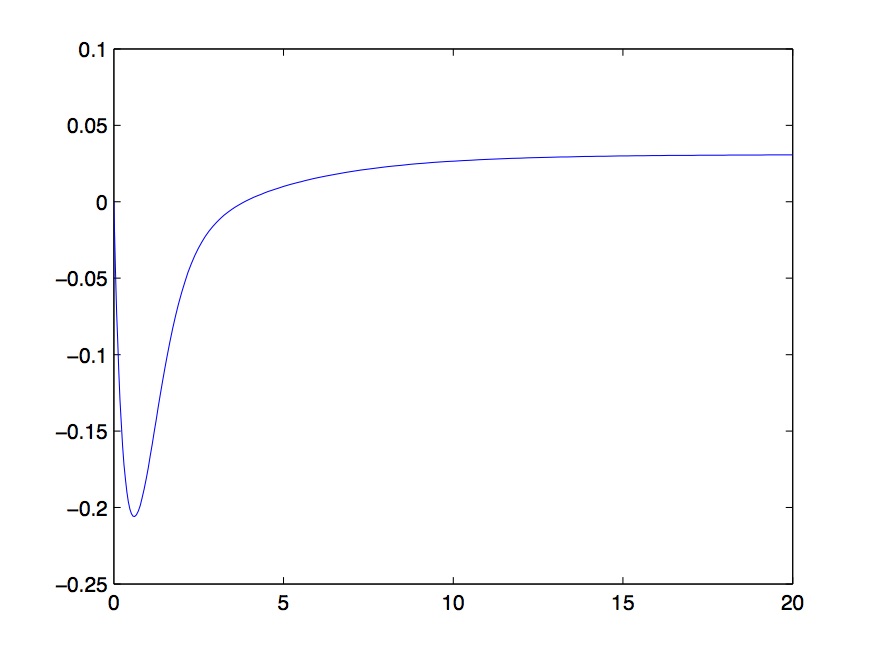}
     \caption{$\hat{y}^{+}_6(\tau)$ for $\delta=0.1$.}
\label{fig:side:a5}
\end{minipage}%
\begin{minipage}[t]{0.5\linewidth}
\centering
\includegraphics[height=0.73\textwidth,width=1\textwidth]{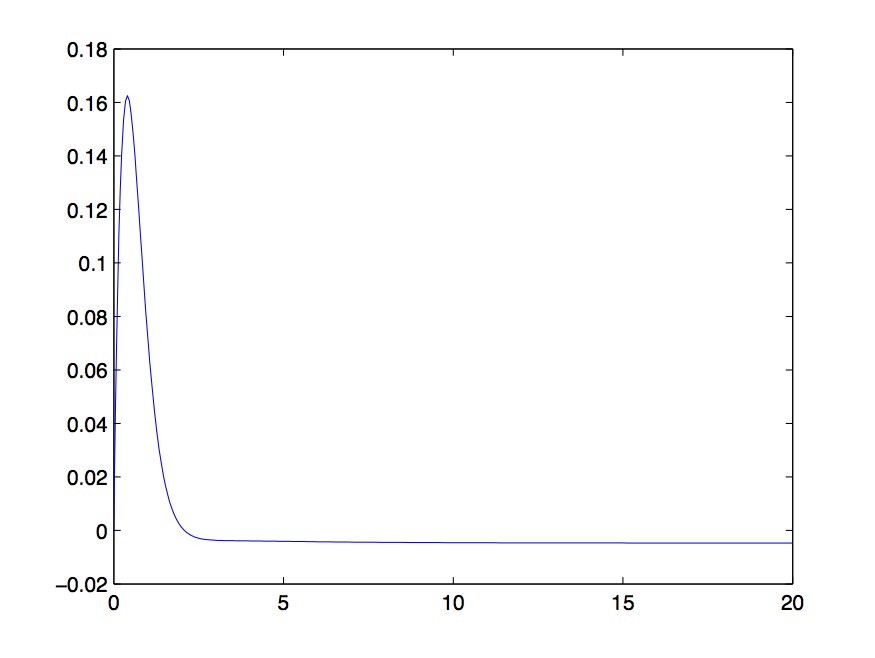}
     \caption{ $\hat{y}^{-}_6(\tau)$ for $\delta=0.1$.}
\label{fig:side:b5}
\end{minipage}
\end{figure}
It is obvious that there is only one zero point in Figure 
\ref{fig:side:a5} and Figure \ref{fig:side:b5}, 
and  we  have computed it for many value
of $\delta\in(0,1/8)$ and for time large as $\tau=1000$.  
All the pictures shows that there is only one zero point.  
This is why we gave Numerical result A.

\begin{figure}[H]
\begin{minipage}[t]{0.5\linewidth}
\centering
\includegraphics[height=0.73\textwidth,width=0.98\textwidth]{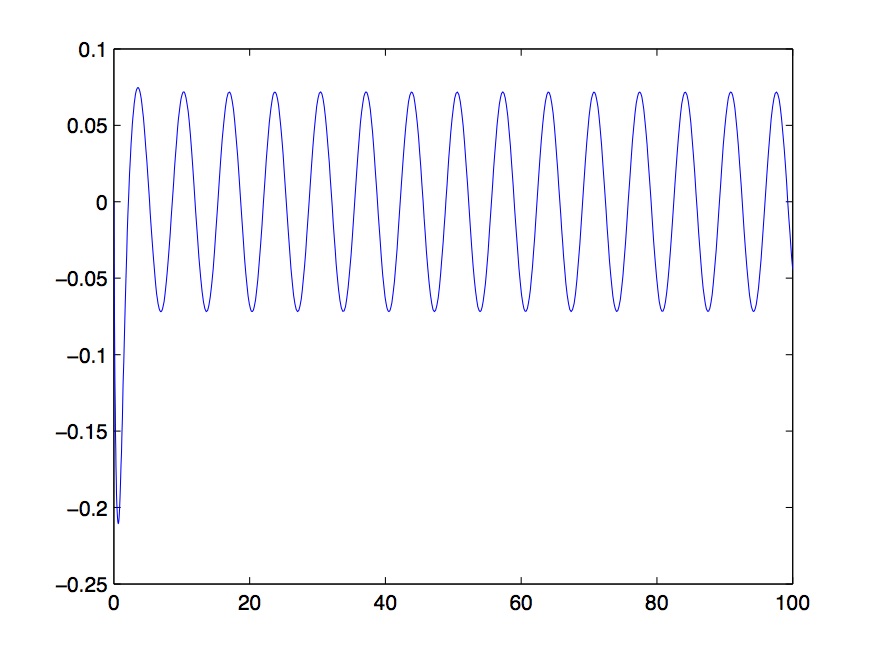}
     \caption{$\hat{y}^{+}_6(\tau)$ for $\delta=1$.}
\label{fig:side:a10}
\end{minipage}%
\begin{minipage}[t]{0.5\linewidth}
\centering
\includegraphics[height=0.73\textwidth,width=1\textwidth]{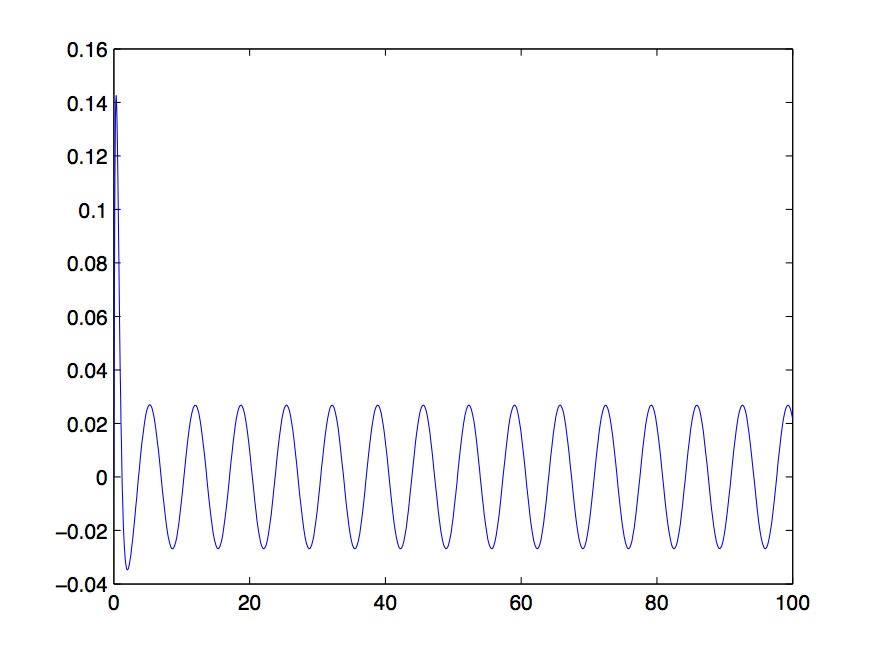}
     \caption{$\hat{y}^{-}_6(\tau)$ for $\delta=1$.}
\label{fig:side:b10}
\end{minipage}
\end{figure}
If $\delta>1/8$, Theorem \ref{thm1.1} shows that the collision index is infinity,  which is  corresponding to
the picture of  Figure \ref{fig:side:a10} and Figure \ref{fig:side:b10}  that the number of zero 
points  growth in direct proportion to the
time.

For the Lagrangian orbits,  Proposition  \ref{prop4.3} shows that the collision index is zero,  
which corresponds to  the following pictures showing    no zero point.
\begin{figure}[H]
\begin{minipage}[t]{0.5\linewidth}
\centering
\includegraphics[height=0.73\textwidth,width=0.98\textwidth]{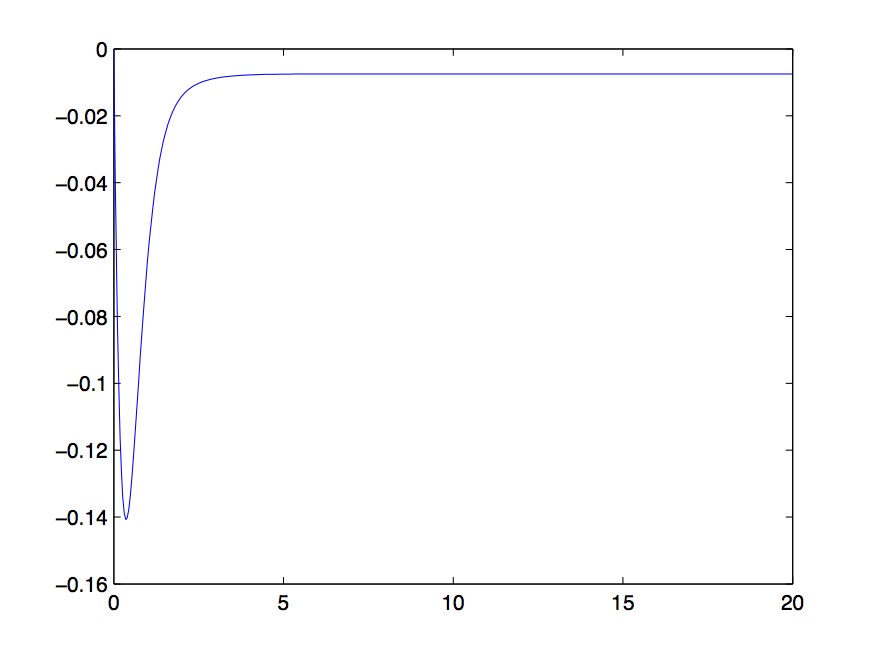}
     \caption{$\hat{y}^{+}_6(\tau)$ for $\beta=6$.}
\label{fig:side:a11}
\end{minipage}%
\begin{minipage}[t]{0.5\linewidth}
\centering
\includegraphics[height=0.73\textwidth,width=1\textwidth]{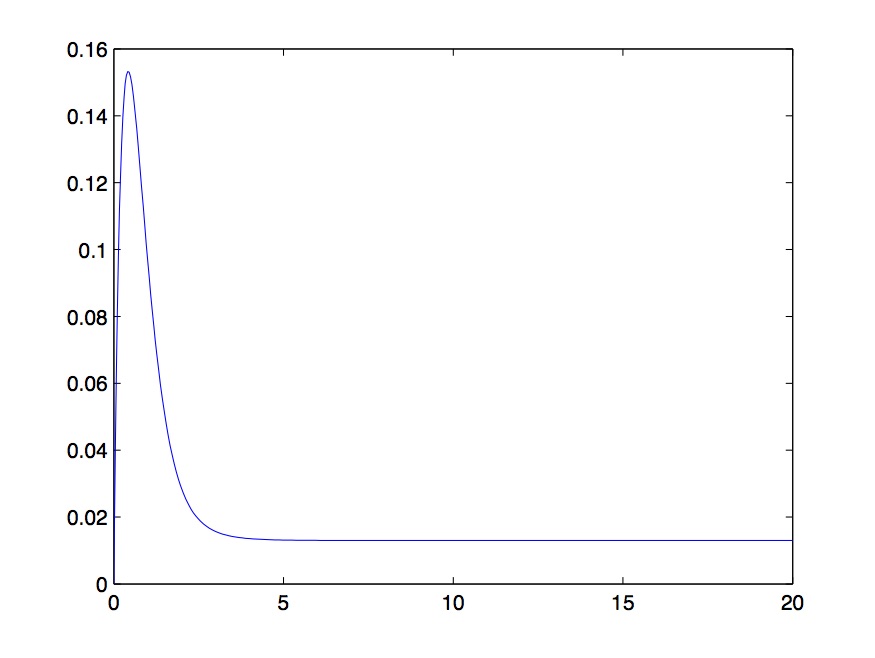}
     \caption{$\hat{y}^{-}_6(\tau)$ for $\beta=6$.}
\label{fig:side:b11}
\end{minipage}
\end{figure}
It is quite difficult in a concrete situation  to establish if an orbit is collision 
non-degenerate. However  if the the collision index depends 
on one  parameter  having  a jump, then there is a collision degenerate point.  
During our computations, we found that the 
Kepler case ($\delta=0$) is collision degenerate. We guess that  
every  nondegenerate central configuration satisfying the condition 
$\lambda_1(R)>-1/8$ is collision nondegenerate.

\medskip

\noindent {\bf Acknowledgements.} The  authors sincerely thank Y.
Long, S. Sun, G. Yu and Q. Zhou for the helpful discussion  for stability
of $n$-body problem.  X. Hu  sincerely thanks S. Terracini,
A. Portaluri, V. Barutello for the discussion on collision index. 
We sincerely thanks the referee for carefully reading 
and for the valuable suggestions that improved  the presentation of the paper. 
We especially  sincerely thanks  A. Portaluri 
for reading and commenting  on a draft of this paper.

\end{document}